\documentclass[10pt]{amsart}
\usepackage{amsmath,amssymb,amsthm}
\usepackage{bbm}
\usepackage{graphicx,tikz,mathtools}
\usepackage[left=2.5cm,right=2.5cm,top=2.5cm,bottom=2.5cm]{geometry}
\usepackage{hyperref}
\usepackage{cite}
\usepackage{mathrsfs} 
\usepackage{verbatim}
\usepackage{caption}
\usepackage{subcaption}
\usepackage{graphicx}
\hypersetup{
	colorlinks=true,
	linkcolor=blue, 
	citecolor=blue,
	filecolor=blue,
	urlcolor=blue,
}

\makeatletter
\@namedef{subjclassname@2020}{\textup{2020} Mathematics Subject Classification}
\makeatother

\newtheorem{theorem}{Theorem}[section]
\newtheorem{proposition}[theorem]{Proposition}
\newtheorem{corollary}[theorem]{Corollary}

\newtheorem{lemma}[theorem]{Lemma}
\newtheorem*{conjecture*}{Conjecture}

\newcommand{\triple}[1]{{\left\vert\kern-0.25ex\left\vert\kern-0.25ex\left\vert #1
        \right\vert\kern-0.25ex\right\vert\kern-0.25ex\right\vert}}

\theoremstyle{definition}

\newtheorem{remark}[theorem]{Remark}

\newcommand{\al}{\alpha}
\newcommand{\be}{\beta}

\newcommand{\ep}{\varepsilon}

\newcommand{\la}{\lambda}

\newcommand{\te}{\theta}
\newcommand{\vp}{\varphi}

\newcommand{\La}{\Lambda}

\def\RR{\mathbb{R}}

\def\HH{\mathbb{H}}

\newcommand{\cD}{{\mathcal D}}

\newcommand{\cF}{{\mathcal F}}

\newcommand{\cJ}{{\mathcal J}}
\newcommand{\cL}{{\mathcal L}}

\newcommand{\cN}{{\mathcal N}}

\newcommand{\cX}{{\mathcal X}}

\newcommand{\pd}{\partial}
\newcommand\minus\backslash

\newcommand\lan\langle
\newcommand\ran\rangle

\newcommand{\dd}{{\mathrm d}}

\newcommand{\norm}[1]{\left\lVert#1\right\rVert}

\renewcommand\leq\leqslant
\renewcommand\geq\geqslant

\numberwithin{equation}{section}

\newcommand{\R}{\mathbb R}
\newcommand{\C}{\mathbb C}
\newcommand{\D}{\mathcal D}
\newcommand{\pv}{\operatorname{p.v.}}
\newcommand{\mi}{\int_{(0,\infty)}}

\allowdisplaybreaks[2]
\begin{document}
 
\title[H\"older estimates for the CCF Equation]{H\"older estimates up to the singular time\\ for the C\'ordoba--C\'ordoba--Fontelos equation}

\author[\'A. Castro]{\'{A}ngel Castro}
\address{    
\newline
\textbf{{\small \'{A}ngel Castro}}
\vspace{0.15cm}
\newline \indent Instituto de Ciencias Matem\'aticas, Consejo Superior de Investigaciones Cient\'\i ficas, 28049 Madrid, Spain}
\email{angel\_castro@icmat.es}

\author[A. Enciso]{Alberto Enciso}
\address{   \vspace{-0.4cm} 
\newline
\textbf{{\small Alberto Enciso}} 
\vspace{0.15cm}
\newline \indent Instituto de Ciencias Matem\'aticas, Consejo Superior de Investigaciones Cient\'\i ficas, 28049 Madrid, Spain}
\email{aenciso@icmat.es}

\author[A. J. Fern\'andez]{Antonio J.\ Fern\'andez}
\address{ \vspace{-0.4cm}
\newline 
\textbf{{\small Antonio J. Fern\'andez}} 
\vspace{0.15cm}
\newline \indent Departamento de Matem\'aticas, Universidad Aut\'onoma de Madrid, 28049 Madrid, Spain}
 \email{antonioj.fernandez@uam.es}

%
%

\begin{abstract}
We prove that the solution to the C\'ordoba--C\'ordoba--Fontelos equation
\(\theta_t-u\theta_x=0\), \(u=H\theta\), on the real line with analytic
initial datum \(\theta_0(x)=(1+x^2)^{-1}\) develops a finite-time
singularity while remaining uniformly bounded in \(C^{2/5}\). We also prove that, as the solution approaches the
singular time, the \(C^{2/3}\) seminorms of both \(\theta\) and \(u\)
diverge. The main ingredient is an integral representation in terms of a
probability measure that is preserved by the nonlinear evolution. From
this representation we derive a pointwise quadratic estimate for the key
nonlinear term in the equation for \(u\), using a comparison of positive
kernels. Together with evolution identities for the moments of the
measure, this estimate yields both the uniform H\"older bound and the
divergence of the higher seminorms. \vspace{-0.5cm}
\end{abstract}
 
\maketitle

\section{Introduction}

We consider the C\'ordoba--C\'ordoba--Fontelos equation
\begin{equation} \label{eq:CCF} \tag{CCF}
\left\{
\begin{aligned}
\, & \te_t - u \te_x = 0 \quad && \textup{in } \R \times (0,T)\,, \\
& u = H \te && \textup{in } \R \times [0,T)\,,\\
& \te(0,\cdot) = \te_0 && \textup{in } \RR\,.
\end{aligned}
\right.
\end{equation}
Here, $H$ is the Hilbert transform with kernel $1/(\pi(x-y))$, so the
transport velocity is $-u$, and $\te_0$ is a bounded, smooth, decaying function.  \eqref{eq:CCF} belongs to a family of one-dimensional nonlocal
transport models studied by Baker, Li and Morlet~\cite{BLM} in connection
with vortex-sheet dynamics. C\'ordoba, C\'ordoba and Fontelos~\cite{CCF}
proved finite-time singularity formation for smooth,
nonnegative, even, compactly supported data with a maximum at the
origin (see also \cite{CCF2}). Their proof uses a weighted nonlinear inequality for the
Hilbert transform to force concentration. Kiselev~\cite{Kiselev}
developed a real-variable approach under symmetry and monotonicity
assumptions, and Silvestre and Vicol~\cite{SV} gave four blow-up
arguments, based on a nonlocal identity, telescoping estimates,
a De Giorgi-type iteration, and barriers. Finally, Li and Rodrigo~\cite{LR2020} obtained further pointwise Hilbert-transform inequalities, an elementary proof of the weighted inequality of~\cite{CCF}, and weighted and nonlinear extensions.

All these results establish loss of smoothness in finite time, but leave open the finer question of what regularity persists. Understanding precisely which norms blow-up when a singularity forms, and which remain uniformly bounded, is a central question in the analysis of nonlinear evolution equations. For nonlocal transport equations such as \eqref{eq:CCF},
identifying this regularity threshold requires a quantitative information on the concentration mechanism that
a blow-up criterion alone does not provide. Silvestre and
Vicol~\cite[Section~7]{SV} proposed an a priori $C^{1/2}$ estimate at positive times for bounded solutions and formulated a corresponding conjecture for vanishing-viscosity limits. They also proved a stationary counterpart under symmetry and monotonicity assumptions. However, controlling the evolution up to the singular time requires information that a stationary cusp or a blow-up criterion alone do not provide.

Our main result gives both upper bounds and divergence estimates for
H\"older norms of the solution with the explicit, analytic, decaying datum
$\theta_0(x)=(1+x^2)^{-1}$. The argument follows this solution throughout
its maximal smooth lifespan and does not assume self-similarity.

\begin{theorem}\label{thm:intro}
Let $\theta_0(x) = (1+x^2)^{-1}$. Then, there exists a unique maximal smooth solution to \eqref{eq:CCF} which, for every integer $k \geq 0$, satisfies
$
\te, u \in C([0,T); H^k(\RR)).
$
Moreover, it follows that:
\begin{itemize}
\item[(i)] $4/3 \leq T \leq 2.$ \smallbreak
\item[(ii)] For every $0 < \alpha \leq 2/5$,
$$
\sup_{0 \leq t < T} \|\te(t,\cdot)\|_{C^{\al}(\RR)} + \sup_{0 \leq t < T} \|u(t,\cdot)\|_{C^{\al}(\RR)} < \infty.
$$
\item[(iii)] For every $2/3 \leq \be < 1$, 
$$
\lim_{t \uparrow T} [\te(t,\cdot)]_{C^{\be}(\RR)} = \lim_{t \uparrow T} [u(t,\cdot)]_{C^{\be}(\RR)} = + \infty.
$$
In particular, it follows that
$$
\lim_{t \uparrow T} \|\te_x(t,\cdot)\|_{L^{\infty}(\RR)} = + \infty\,.
$$
\end{itemize}
\end{theorem}

\begin{remark}\label{lem:parity} 
If $\theta_0$ is even, then, for every $t \in (0,T)$, the unique solution to \eqref{eq:CCF}, $\theta(t,\cdot)$, is even. Likewise, the corresponding field $u(t,\cdot)=H\theta(t,\cdot)$ is odd, as is the transport velocity $-u(t,\cdot)$. 
\end{remark}

The quantitative version of this result in Theorem~\ref{thm:main} gives explicit estimates for $\|u_x(t,\cdot)\|_{L^\infty}$, $\|\theta_x(t,\cdot)\|_{L^\infty}$, and the endpoint
H\"older seminorms. However, it does not identify the optimal H\"older exponent. In particular, the theorem neither proves nor disproves the $C^{1/2}$ conjecture \cite{SV}. In the next two subsections we introduce the main ideas behind the proof of Theorem \ref{thm:intro} and provide a sketch of the proof. 

\subsection{Sketch of the proof} The first key idea is to find an integral representation that is preserved by the nonlinear equation and makes its nonlocal terms explicit.
In this direction, we prove that the representation
\begin{equation}\label{eq:intro-mixture}
 \theta(t,x)=\int_{(0,\infty)}\frac{a^2}{a^2+x^2}\,\dd\mu_t(a),
 \qquad \mu_t\geq0,\qquad \mu_t((0,\infty))=1,
\end{equation}
with a Borel probability measure on the positive half-line, is preserved for as long as the smooth maximal solution exists, provided the initial datum
has the Sobolev regularity and moment bounds specified
in Proposition~\ref{prop:local}. At each time, the representing
probability measure is unique, and the nonlocal field $u=H\theta$ is
recovered from the same measure, namely
\begin{equation}\label{eq:intro-u-mixture}
 u(t,x)=H\theta(t,x)
 =\int_{(0,\infty)}\frac{ax}{a^2+x^2}\,\dd\mu_t(a).
\end{equation}
The measure records the distribution of spatial scales in the solution.
Positivity is explicit in the representation, and the moments control
derivative bounds and concentration.

The datum of Theorem~\ref{thm:intro} corresponds to $\mu_0=\delta_1$, but it is worth pointing out that the representation applies to all initial data satisfying the assumptions in Proposition \ref{prop:local}.

After the change of variable $r=x^2$, representation~\eqref{eq:intro-mixture}
is a normalized Stieltjes representation, whose characterization is
classical; see~\cite{Berg,KarpPrilepkina,SSV}. To prove preservation,
we first freeze the field $u$. In the variable $r=x^2$, the backward
characteristics extend holomorphically and preserve the half-plane sign
conditions that characterize Stieltjes functions. This gives the
representation for the linear transport step. A nonlinear iteration,
with control of the first and inverse first moments, then propagates it
throughout the maximal lifespan $[0,T)$.

Let us now discuss how the representation described above enters the proof. Roughly speaking, it gives a useful sign to the nonlocal term in the
equation for $u$, which is very hard to capture otherwise. This field satisfies the Burgers type equation
\[
 u_t-uu_x=-\mathcal N(u),\qquad
 \mathcal N(u):=H(uHu_x)+uu_x.
\]
For every smooth function with our probability representation, we prove the pointwise bound
\begin{equation}\label{eq:intro-quadratic}
 \mathcal N(u)(t,x)\geq\frac{u(t,x)^2}{4x}, \quad \textup{for all } (t,x) \in [0,T) \times (0,\infty).
\end{equation}
Let us stress that, for a general function, the nonlocal expression on the left has no
evident pointwise sign. However, with the representation above, it is an explicit quadratic integral against the positive measure
$\mu_t\otimes\mu_t$. The comparison in~\eqref{eq:intro-quadratic} reduces
to the elementary identity
\[
 2\bigl(a^2+b^2+2x^2\bigr)-(a+b)^2
 =(a-b)^2+4x^2\geq0. 
\]
See Lemma~\ref{lem:force} for the details. The nonlocal nonlinear estimate is thus reduced to a
comparison of two elementary kernels. Let us point out that the same strategy may be useful
for other equations that preserve a suitable positive representation.

Finally, let us explain how this lower bound and the moments of the measure give the regularity bounds we were looking for. The inverse first moment
\[
 m(t):=\int_{(0,\infty)}a^{-1}\,\dd\mu_t(a)
      =u_x(t,0)=\Lambda\theta(t,0)=\|u_x(t,\cdot)\|_{L^\infty},
\]
measures the strength of the compression at the origin. Evaluating
the differentiated equation for $u$ there yields the Riccati type ODE and bounds
\begin{align}\label{infinite}
 m'(t)=\frac12m(t)^2+I(t),\qquad
 I(t):=\iint_{(0,\infty)^2}\frac{\dd\mu_t(a)\,\dd\mu_t(b)}{(a+b)^2},
 \qquad 0<I(t)\leq\frac14m(t)^2.
\end{align}
These bounds determine the stated lifespan interval and give
$m(t)\asymp(T-t)^{-1}$.

To obtain the uniform H\"older bound, we use the transported deficit
$P(t,x):=1-\te(t,x)$. With the material derivative
$D:=\partial_t-u\partial_x$, we have $DP=0$, $Dx=-u$, and
$Du=-\mathcal N(u)$. Consequently,~\eqref{eq:intro-quadratic} gives
\[
 D\!\left(\frac{u(t,x)}{x^{1/4}P(t,x)^{3/8}}\right)\leq0, \quad
 \quad \textup{for all } (t,x) \in [0,T) \times (0,\infty).
\]
Moreover, the initial quotient is at most one. Combining the resulting bound
$u(t,x)\leq x^{1/4}P(t,x)^{3/8}$ with the kernel inequality $P_x(t,x)\leq u(t,x)/x$ yields $P(t,x)\lesssim x^{2/5}$ and $u(t,x)\lesssim x^{2/5}$. The corresponding derivative estimates give the desired global $C^{2/5}$ control.

For the opposite bound, we note that the inverse second moment is encoded by the
curvature $K(t):=-\theta_{xx}(t,0)$, which satisfies the ODE $K'(t)=2m(t)K(t)$. Interpolation between $m$ and $K$ gives lower bounds for the H\"older seminorms. Keeping only the Riccati type inequalities would prove
divergence for $\beta>2/3$. To include $\beta=2/3$, we retain the
positive remainder $I(t)$. Then, a layer-cake argument gives a logarithmic lower bound for it. This improves the comparison between $K$ and $m$ and yields divergence also at the endpoint.

\subsection{Related results and further questions}

Now that we have an idea of how the proof of Theorem \ref{thm:intro} goes, let us compare this result with the available literature and propose some possible extensions. 

Differentiating~\eqref{eq:CCF} and setting $\omega=\theta_x$
gives
\[
 \omega_t-u\omega_x=u_x\omega,\qquad u_x=H\omega.
\]
Thus, the differentiated equation is the $a=-1$ member of
the generalized Constantin--Lax--Majda family
\[
 \omega_t+a u\omega_x=u_x\omega,\qquad u_x=H\omega,
\]
also known as the Okamoto--Sakajo--Wunsch model; see~\cite{CLM,OSW}. The first author and C\'ordoba~\cite{CC2010} extended the previously known finite-time blow-up mechanism for \eqref{eq:CCF} to all $a<0$ by using a different proof of the differential inequality $m'\geq \frac{1}{2}m^2$, obtained from \eqref{infinite}. We do not assert that the probability measure representation is preserved for other members of this family. 

So far, the most common approach in the literature to study singularities for these problems has been the use of self-similar profiles. 

In this direction, Elgindi and Jeong~\cite{EJ2020}
constructed smooth self-similar blow-up profiles for
sufficiently small $|a|$, and H\"older profiles for arbitrary
$a$, provided the H\"older exponent is sufficiently small.
Elgindi, Ghoul and Masmoudi~\cite{EGM2021} proved stability
of these blow-up regimes. For \eqref{eq:CCF}, these constructions
start from nonsmooth vorticities, corresponding to initial
data with just $C_{\rm loc}^{1,\alpha}$ regularity. They also show
that H\"older regularity of the associated velocity can
persist up to the singular time.  

Huang, Qin, Wang and Wei~\cite{HQWW} constructed exact
self-similar solutions with interiorly smooth profiles
for all $a\leq1$. Their $a=-1$ profiles are smooth on
$\R$ and satisfy
$|\Omega(x)|\sim C|x|^{-1/c_l}$, with $1<c_l<3$.
In particular, their primitives are unbounded at infinity. This
existence theorem for the differentiated equation therefore does not provide the H\"older estimates considered here for a smooth, bounded initial datum.

Numerical and asymptotic work~\cite{EF2020,LSS,WLG2023,
Discovery2025,WLLB2025} suggests that different H\"older
thresholds may occur. In particular, the unstable \eqref{eq:CCF} self-similar profiles predict scalar cusp exponents below $1/2$.
Realizing these singularities from smooth, bounded,
decaying data requires an additional argument controlling
the evolution. Huang, Tong and Wang~\cite{HTW2026} also
investigated singular profiles and two-scale dynamics.
Their convergence theorem concerns $a=0$, while the
corresponding \eqref{eq:CCF} dynamics are studied numerically.

A natural next question is whether the gap between
$2/5$ and $2/3$ in Theorem~\ref{thm:intro} can be narrowed.
For the related Hou--Luo model, Chen, Hou and
Huang~\cite{CHH2022} established a H\"older threshold
through a computer-assisted stability argument. Here, the probability representation suggests a route to
computer-assisted proofs with sharper H\"older bounds. A concrete goal is to find $\gamma\in(0,1)$ and $0 < \varepsilon \ll 1$ such that
\[
\begin{aligned}
 &\sup_{0\leq t<T}\|\theta(t,\cdot)\|_{C^\alpha}<\infty
 &&\text{for }0<\alpha\leq\gamma-\varepsilon,\\
 &\lim_{t\uparrow T}[\theta(t,\cdot)]_{C^\beta}=+\infty
 &&\text{for }\gamma+\varepsilon\leq\beta<1.
\end{aligned}
\]
For the initial datum in Theorem~\ref{thm:intro}, we expect that such a
certification would give $\gamma-\varepsilon>1/2$, although this is not
established here. For other smooth initial data, we expect H\"older
thresholds below $1/2$, as suggested by the unstable profiles discussed
above. Certifying either assertion requires control of the dynamics
approaching $T$, not merely an accurate numerical profile or estimates
on a fixed compact subinterval of $[0,T)$.

A separate possible extension concerns compactly supported data.
For an even cutoff $\chi\in C_c^\infty(\R)$ with $0\leq\chi\leq1$,
$\chi=1$ on $[-1,1]$, $\chi=0$ outside $[-2,2]$, and $\chi'\leq0$
on $(0,\infty)$, set
\[
 \theta_{0,R}(x):=\frac{\chi(x/R)}{1+x^2}.
\]
We expect that suitable tail estimates could extend the conclusions of
Theorem~\ref{thm:intro} to these data for sufficiently large $R$, with
$T$ replaced by a lifespan $T_R\in(0,2)$. This would require an additional
argument: any function with representation~\eqref{eq:intro-mixture} is strictly
positive on $\R$, so the truncated datum does not admit this representation. No such
extension is proved in this paper.

\subsection{Organization of the paper}
Section~\ref{sec:preliminaries} develops the properties of the probability
measure representation and proves the representation lemma for frozen
transport. Section~\ref{sec:invariance} establishes the nonlinear
construction and preservation of the representation.
Section~\ref{sec:forcing} proves the forcing identity and the quadratic
estimate. Section~\ref{sec:main-proof} states and proves the quantitative
version of Theorem~\ref{thm:intro}. The transport estimate used in the
construction is recalled in Appendix~\ref{app:energy}.

\subsection{Declaration on the use of AI tools}
AI tools, specifically ChatGPT, were used during the preparation of this
work to assist with mathematical exploration, intermediate
calculations, bibliographic searches, and the drafting and revision
of the manuscript. Their outputs were treated as proposals to be
checked, not as evidence of correctness. The authors checked the
arguments and references and take full responsibility for the results
and their presentation. The proofs in this paper 
do not use computer assisted estimates.

\subsection{Notation}\label{subsec:conventions}
The Hilbert transform and the half-Laplacian are
\begin{equation}\label{eq:H}
 Hf(x):=\frac1\pi\pv\int_{\R}\frac{f(y)}{x-y}\dd y \quad \textup{and} \quad \Lambda f(x) := (-\pd_x^2)^{\frac12}f(x) = \frac{1}{\pi} \pv \int_{\RR} \frac{f(x)-f(y)}{(x-y)^2} \dd y\,.
\end{equation}
Thus, using the convention $\widehat f(\xi)=\int_\R e^{-ix\xi}f(x)\,dx$, their multipliers are
$-i\operatorname{sgn}\xi$ and $|\xi|$, respectively. Moreover, with this convention
$$
H^2 = -I \quad \textup{in } L^2, \qquad H \pd_x = \pd_x H = \La.
$$
Throughout the paper, we use the nonlocal field $u$ and the material derivative $D$, with
\begin{equation}\label{eq:global-signs}
 u=H\theta,\qquad Hu=-\theta,\qquad
 D=\partial_t-u\partial_x.
\end{equation}
Thus the characteristic velocity in \eqref{eq:CCF} is $-u$.
We use the notation $f_x := \pd_x f$ and $f_t := \pd_t f$. Unless a domain is specified, all norms are spatial norms on $\R$. For $0<\alpha<1$, we write
\[
 [f]_{C^\alpha}:=\sup_{x\ne y}\frac{|f(x)-f(y)|}{|x-y|^\alpha},
 \qquad \norm f_{C^\alpha}:=\norm f_{L^\infty}+[f]_{C^\alpha}.
\]

\section{Probability measure representation and a key transport lemma}\label{sec:preliminaries}

For $a>0$, we set and fix
\begin{equation}\label{eq:kernels}
 p_a(x):=\frac{a^2}{a^2+x^2} \quad \textup{and} \quad 
 h_a(x):=\frac{ax}{a^2+x^2}.
\end{equation}
Thus, it follows that
\begin{equation}\label{eq:Hkernels}
 Hp_a=h_a,\qquad
 \norm{p_a}_{L^2}=\norm{h_a}_{L^2}=\Big(\frac{\pi a}{2}\Big)^{1/2}.
\end{equation}
We consider functions with the integral representation
\begin{equation} \label{def:mixture}
 f_\mu(x)=\mi p_a(x)\dd\mu(a),
\end{equation}
where $\mu$ is a Borel probability measure on $(0,\infty)$. The \textit{first} and
\textit{inverse first moments} of the measure are denoted by
\[
 A_\mu=\mi a\dd\mu(a),\qquad m_\mu=\mi a^{-1}\dd\mu(a).
\]

Having this notation at hand, we start with a preliminary technical lemma whose proof is elementary.   
 
\begin{lemma}\label{lem:mixture}
Let $f = f_\mu$ be as in \eqref{def:mixture}. Then
\begin{align*}
 A_\mu=\frac2\pi\int_0^\infty f(x)\dd x
\quad \textup{and} \quad 
 m_\mu=\frac2\pi\int_0^\infty\frac{1-f(x)}{x^2}\dd x.
\end{align*}
These identities hold even when the integrals are infinite. If $A_\mu<\infty$,
then $f\in L^2$ and
\begin{equation} \label{eq:umixture}
 Hf(x)=\mi h_a(x)\dd\mu(a).
\end{equation}
If also $m_\mu<\infty$, then $f$ and $Hf$ are continuously differentiable and,
with $b(r):=\mi a/(a^2+r)\dd\mu(a)$,
\begin{align*}
 |(Hf)'(x)|&\le b(x^2)\le m_\mu,
 &(Hf)'(0)&=m_\mu,
\\
 \norm{f'}_{L^\infty}&\le m_\mu,
 &\norm{Hf}_{L^\infty}&\le\tfrac12.
\end{align*}
\end{lemma}

\begin{proof}
The proof of the first two identities is direct calculus. In order to check the third, we use a duality argument. Before doing so, we set  $g(x) := \int_{(0,\infty)} h_a(x) \dd\mu(a)$ and point out that $f$ and $g$ are both continuous functions. Moreover, note that $A_\mu < \infty$ implies $\int_{(0,\infty)} \sqrt{a} \dd\mu(a) \leq \sqrt{A_\mu}$. This allows the subsequent use of Fubini's theorem. Let $\vp \in L^2$ be fixed but arbitrary. Using that $H^{*} = - H$ on $L^2$ and Fubini, we get that
\begin{align*}
 \langle Hf,\varphi\rangle_{L^2}
 &=-\langle f,H\varphi\rangle_{L^2}
   =-\int_{(0,\infty)}\langle p_a,H\varphi\rangle_{L^2}\dd\mu(a)\\
 &=\int_{(0,\infty)}\langle Hp_a,\varphi\rangle_{L^2}\dd\mu(a)
   =\int_{(0,\infty)}\langle h_a,\varphi\rangle_{L^2}\dd\mu(a)
   =\langle g,\varphi\rangle_{L^2}.
\end{align*}
Therefore $Hf=g$ in $L^2$. Moreover, this holds pointwise when $Hf$ is represented by the continuous function $g$. In particular, it follows that
$\norm{Hf}_{L^\infty}=\norm{g}_{L^\infty} \le 1/2$.

Assume in addition that $m_\mu < \infty$. 
Direct differentiation yields
\[
 h_a'(x)=\frac{a(a^2-x^2)}{(a^2+x^2)^2},\qquad
 p_a'(x)=-\frac{2a^2x}{(a^2+x^2)^2}.
\]
Moreover, $|h_a'(x)|\le a/(a^2+x^2)\le1/a$, $h_a'(0)=1/a$, and $|p_a'(x)|\le1/a$. Since $m_\mu<\infty$, these bounds justify differentiation under the integral sign and prove the continuous differentiability and all the claimed derivative bounds.
\end{proof}

\begin{remark}\label{rem:moment}
Assume $A_\mu<\infty$. Even without initially assuming $m_\mu<\infty$, \eqref{eq:umixture} gives
\[
 \lim_{x\downarrow0}\frac{Hf(x)}x
 =\lim_{x\downarrow0}\mi\frac{a}{a^2+x^2}\dd\mu(a)
 =\mi a^{-1}\dd\mu(a)
\]
by monotone convergence. If $Hf$ is differentiable at zero, this identifies
the moment with $(Hf)'(0)$ and proves its finiteness. This observation
avoids assuming an interchange of differentiation and integration in order
to establish that the inverse moment exists.
\end{remark}

The rest of this section is devoted to proving a key ODE transport lemma that will be crucial in the proof of our main result. However, before stating and proving such a result, we record a classical representation of \textit{Stieltjes} functions. Setting $\D:=\C\setminus(-\infty,0]$ and $\HH:=\{z:\operatorname{Im}z>0\}$, we prove the following.

\begin{lemma}\label{lem:stieltjes}
Let $q$ be holomorphic on $\D$, nonnegative on $(0,\infty)$, and satisfy
$\operatorname{Im}q(z)\le0$ for $z\in\HH$. Then
\begin{equation*} 
 q(z)=c+\int_{[0,\infty)}\frac{\dd\rho(s)}{z+s},
 \qquad c\ge0,\quad
 \int_{[0,\infty)}\frac{\dd\rho(s)}{1+s}<\infty,
\end{equation*}
for a unique positive measure $\rho$ and constant $c$. If, in addition,
$\lim_{r \downarrow 0} q(r)=1$ and $\lim_{r \to \infty} q(r)=0$, then
\begin{equation*} 
 q(z)=\mi\frac{a^2}{a^2+z}\dd\nu(a)
\end{equation*}
for a unique probability measure $\nu$ on $(0,\infty)$.
\end{lemma}

\begin{proof}
The unnormalized representation follows for instance from
\cite[Theorem 13]{KarpPrilepkina}. Uniqueness of the representing measure and constants is recalled in \cite[p. 16]{SSV}.
For the normalized representation, dominated convergence as $r\to\infty$ gives $c=0$. The finite limit as $r\downarrow0$ then implies $\rho(\{0\})=0$ and, by monotone convergence,
\[
 \int_{(0,\infty)}s^{-1}\dd\rho(s)=1.
\]
Define $\nu$ to be the pushforward of the probability measure $s^{-1}\dd\rho(s)$ under $s\mapsto\sqrt{s}$. Substituting $s=a^2$ in the representation of $q$ gives the desired formula. This change of measure is invertible, so uniqueness follows from that of $\rho$. See also \cite{Berg}.
\end{proof}

\begin{remark} To a function $f_\mu$ as in \eqref{def:mixture}, we associate the complex function 
$$
Q_{\mu}(z) := \int_{(0,\infty)} \frac{a^2}{a^2+z} \dd\mu(a), \quad z \in \cD.
$$
We call $Q_\mu$ the \textit{normalized Stieltjes function} associated with $f_\mu$. It is worth pointing out that $Q_\mu$ is nothing but the holomorphic extension to $\cD$ of $r \mapsto f_\mu(\sqrt{r})$. 
\end{remark}

\begin{lemma}
\label{lem:frozen}
Let $(\mu_t)_{0\le t\le T}$ be a weakly measurable family of
Borel probability measures on $(0,\infty)$, with moments
\[
 A(t):=\mi a\,\dd\mu_t(a),
 \qquad
 m(t):=\mi a^{-1}\,\dd\mu_t(a).
\]
Assume that $A,m\in L^1(0,T)$ and that
$A_0:=A(0)$ and $m_0:=m(0)$ are finite, and set
\[
 F(t,z):=2z\mi\frac{a}{a^2+z}\,\dd\mu_t(a),
 \qquad
 Q_0(z):=\mi\frac{a^2}{a^2+z}\,\dd\mu_0(a).
\]
Then, the transport problem
\begin{equation}\label{eq:frozenPDE}
 Q_t=F(t,r)Q_r,
 \qquad Q(0,r)=Q_0(r),
 \qquad r>0,
\end{equation}
has a unique characteristic solution on $[0,T]$ and the equation holds for almost every $t$.
For every $t\in[0,T]$, this solution extends holomorphically
to $\D$ and admits the representation
\begin{equation}\label{eq:composition}
 Q(t,z)=\mi\frac{a^2}{a^2+z}\,\dd\nu_t(a),
 \qquad z\in\D,
\end{equation}
where $\nu_t$ is a unique Borel probability measure on
$(0,\infty)$, with $\nu_0=\mu_0$.
Writing $L(t):=\int_0^t m(s)\,\dd s$, its moments satisfy
\begin{align}  \label{eq:momentAflow}
 e^{-L(t)}A_0
 &\le \mi a\,\dd\nu_t(a)\le A_0 \quad \textup{and} \quad
 m_0 
 \le \mi a^{-1}\,\dd\nu_t(a)\le e^{L(t)}m_0.
\end{align}
In particular, setting $f(t,x):=Q(t,x^2)$ for $x\ne0$
and $f(t,0):=1$, each $f(t,\cdot)$ has the representation in \eqref{def:mixture}, and
\begin{equation} \label{eq:ftcomparison}
 0\le f(t,x)\le f(0,x), \quad (t,x) \in [0,T] \times \RR\,.
\end{equation}
\end{lemma}

\begin{proof}
First of all, we set
\begin{equation}\label{eq:BF}
 B(t,z):=\mi\frac{a}{a^2+z}\,\dd\mu_t(a),
 \qquad F(t,z)=2zB(t,z).
\end{equation}
Then, note that for each compact $K \subset \cD$, there exists a constant $c_K > 0$ such that
$$
|a^2+z| \geq c_K (1+a^2)\,, \quad a \in \RR, \ z \in K\,.
$$
Combining this inequality with the probability normalization, one can easily show that $F(\cdot,z)$ is measurable for every $z \in K$, $F(t,\cdot)$ is holomorphic on $\cD$ for every $0 \leq t \leq T$, and $F$ is jointly measurable in $(t,z)$. Moreover, it follows that
\begin{equation}\label{eq:frozen-local-bounds}
 \sup_{s\in[0,T]}\sup_{z\in K}
 \bigl(|F(s,z)|+|F_z(s,z)|\bigr)\le C_K.
\end{equation}

Next, we consider the backwards characteristic equation
\begin{equation} \label{eq:backwards}
Z'(\sigma) = F(t-\sigma, Z(\sigma)),\quad  Z(0) = z, \quad 0 \leq \sigma \leq t \leq T.
\end{equation}
For real $z>0$, this describes the characteristic of
\eqref{eq:frozenPDE} followed backwards from position $z$
at time $t$. We shall first prove that it is well defined for every $z\in\D$, and then set
\[
 \Phi_t(z):=Z(t),\qquad Q(t,z):=Q_0(\Phi_t(z)).
\]

Having \eqref{eq:frozen-local-bounds} at hand, we shall first prove the local existence, uniqueness and holomorphic dependence of a solution to \eqref{eq:backwards}. The integral formulation of \eqref{eq:backwards} is
\[
 Z(\sigma)
 =z+\int_0^\sigma F(t-\tau,Z(\tau))\,\dd\tau.
\]
On a sufficiently short interval, the right-hand side defines
a contraction on continuous curves taking values in a small
closed disk centered at $z$, by \eqref{eq:frozen-local-bounds}.
Its fixed point is the unique local absolutely continuous
characteristic. Moreover, the Picard iterates are holomorphic in the
initial point, and converge locally uniformly. Hence, the
characteristic depends holomorphically on that point. The same local construction gives a continuation criterion: a solution extends past a finite endpoint if its image is contained in a compact subset of $\cD$. 

Once we have this local existence, and the continuation criterion, we are going to prove global continuation. We fix $t \in [0,T]$. First observe that, for $z\in\HH$, direct calculation gives
\begin{equation}\label{eq:frozen-imaginary}
 \operatorname{Im}F(s,z)
 =2\operatorname{Im}z\mi\frac{a^3}{|a^2+z|^2}\,\dd\mu_s(a)\ge0.
\end{equation}
Hence, if $\eta=\operatorname{Im}z>0$, then
$\operatorname{Im}Z(\sigma)\ge\eta$ and
\[
 |F(s,Z)|\le\frac{2A(s)}{\eta}|Z|.
 \]
Moreover, combining these estimates with Gronwall's inequality, we get that
\[
 |Z(\sigma)|\le |z|\exp\left(\frac2\eta\int_0^t A(s)\,\dd s\right).
\]
Thus, the trajectory stays in a bounded closed subset of
$\{\operatorname{Im}z\ge\eta\}\subset\D$. The continuation
criterion previously described proves existence up to $\sigma=t$ and preservation of
the upper half-plane. Since
$F(s,\overline z)=\overline{F(s,z)}$, conjugating a trajectory and
using uniqueness proves the corresponding assertion in the lower
half-plane. In order to prove the continuation on the whole of $\cD$ we are just missing to show the corresponding result in the right half-plane $\{z: {\rm Re }\, z > 0\}$. 

For a starting point with $\operatorname{Re}z=x_0>0$, one can easily check that the real part of the trajectory is nondecreasing and therefore remains at least $x_0$.
Also $|a^2+Z|\ge a^2$ in this half-plane. Hence,
\[
 |F(s,Z)|\le2m(s)|Z|,
 \qquad
 |Z(\sigma)|\le|z|\exp\left(2\int_0^t m(s)\dd s\right).
\]
This puts the trajectory in another compact subset of $\D$ and
again proves existence up to $\sigma = t$. The upper, lower and right
half-planes cover $\D$, so every starting point in $\D$ is covered.
On their overlaps the solutions agree by uniqueness. We have now proved the desired continuation. 

Moreover, observe that, for real $z = r>0$, conjugation and uniqueness imply that $Z(\sigma)$
is real. The right-half-plane result makes it positive, and
\[
 0\le\frac{Z'(\sigma)}{Z(\sigma)}
 =2\int_{(0,\infty)}\frac a{a^2+Z(\sigma)}\dd\mu_{t-\sigma}(a)
 \le2m(t-\sigma)
\]
almost everywhere. Integration prevents both reaching
zero and escape to infinity. Thus, $\Phi_t(z)=Z(t)$ is holomorphic
on $\D$, preserves $\HH$ and $(0,\infty)$, and satisfies
\begin{equation}\label{eq:realflow}
 r\le\Phi_t(r)\le e^{2L(t)}r, \quad \textup{for all } r > 0.
\end{equation}

Next, for $r_0 > 0$, we consider the forward real characteristic equation
$$
X_t(t,r_0) = - F(t,X(t,r_0)), \quad X(0,r_0) = r_0.
$$
The forward real flow $X(\cdot,r_0)$ likewise
exists throughout $[0,T]$, since
\begin{equation}\label{eq:frozen-forward-bounds}
 r_0e^{-2L(t)}\le X(t,r_0)\le r_0.
\end{equation}
Moreover,
\[
 F_r(s,r)=2\mi\frac{a^3}{(a^2+r)^2}\,\dd\mu_s(a),
 \qquad 0\le F_r(s,r)\le2m(s),
\]
so its Jacobian satisfies
\begin{equation}\label{eq:frozen-flow-jacobian}
 J(t,r_0):=\partial_{r_0}X(t,r_0)
 =\exp\left(-\int_0^t F_r(s,X(s,r_0))\,\dd s\right)
 \in[e^{-2L(t)},1].
\end{equation}
Consequently $X(t,\cdot)$ is an increasing bijection of
$(0,\infty)$ onto itself. Reversing time in \eqref{eq:backwards}
shows that
\begin{equation}\label{eq:frozen-inverse-flow}
 X(t,\Phi_t(r))=r, \quad \textup{for all } r > 0.
\end{equation}
The local bounds on $F,F_r$ and the lower bound on $J$ make the
flow and its inverse locally Lipschitz in time. The absolutely
continuous chain rule applied to this identity gives
\[
 \partial_r\Phi_t(r)=\frac1{J(t,\Phi_t(r))},
 \qquad
 \partial_t\Phi_t(r)=F(t,r)\partial_r\Phi_t(r)
\]
for almost every $t$ in the second identity. Hence,
$Q(t,r):=Q_0(\Phi_t(r))$ solves \eqref{eq:frozenPDE} on $[0,T]$. Note that this last identity can also be written as $Q(t, X(t,r)) = Q_0(r)$. Since $X(t,\cdot)$ is a bijection of $(0,\infty)$ onto itself,
this determines $Q$ uniquely, showing that \eqref{eq:frozenPDE} has indeed a unique characteristic solution on $[0,T]$.

We now show that the constructed characteristic solution admits the \textit{Stieltjes} representation given in \eqref{eq:composition}. The same compact kernel bound \eqref{eq:frozen-local-bounds} shows that $Q_0$ is holomorphic
on $\D$. Its integral formula and dominated convergence give
\[
 Q_0(r)>0,\qquad Q_0'(r)<0 \textup{ for } r > 0,\qquad
 \operatorname{Im}Q_0(z)\le0 \textup{ for } z\in\HH,
 \qquad \lim_{r \downarrow 0} Q_0(r)=1,\quad \lim_{r \to \infty} Q_0(r)=0.
\]
Since $\Phi_t$ preserves $\HH$ and the positive axis,
$Q_0\circ\Phi_t$ has the same sign properties, and
\eqref{eq:realflow} gives the same endpoint limits.
Lemma~\ref{lem:stieltjes} therefore yields the existence of the unique probability
measure $\nu_t$ in \eqref{eq:composition}. At $t=0$, uniqueness
gives $\nu_0=\mu_0$. We have therefore the desired representation. 

Finally, set $\kappa=e^{L(t)}$. Monotonicity of $Q_0$ and
\eqref{eq:realflow} give
\begin{equation} \label{eq:comparisonft}
 f(0,\kappa x)\le f(t,x)\le f(0,x), \quad \textup{ for } x \in \RR.
\end{equation}
By the two moment identities in Lemma \ref{lem:mixture}, initially
allowing infinite values,
\[
 \mi a\,\dd\nu_t(a)=\frac2\pi\int_0^\infty f(t,x)\,\dd x,
 \qquad
 \mi a^{-1}\,\dd\nu_t(a)
 =\frac2\pi\int_0^\infty\frac{1-f(t,x)}{x^2}\,\dd x.
\]
Integrating the comparison for $f(t,\cdot)$, namely \eqref{eq:comparisonft}, and then the reversed comparison for $1-f(t,\cdot)$ against $1/x^2$, proves
\eqref{eq:momentAflow} after the change
of variables $y=\kappa x$. These bounds also prove finiteness of
both moments and complete the proof.
\end{proof}

\section{Nonlinear construction and preservation of the representation}\label{sec:invariance}

In this section, building on Lemma \ref{lem:frozen}, we prove a representation by a probability measure for the unique solution to \eqref{eq:CCF}. More precisely, we prove the following:

\begin{proposition} \label{prop:local}
Let $f$ have the representation in \eqref{def:mixture}, with
$f\in\bigcap_{k\ge0}H^k$ and finite moments $A_0,m_0$.
Then \eqref{eq:CCF}, with $\theta(0)=f$, has a unique smooth Sobolev solution
on every $[0,\tau]$ with $\tau<m_0^{-1}$.
For each such time, it has a unique probability measure representation
\begin{equation}\label{eq:nonlinearrep}
 \theta(t,x)=\mi p_a(x)\dd\mu_t(a),\qquad
 u(t,x)=H\theta(t,x)=\mi h_a(x)\dd\mu_t(a).
\end{equation}
The measures are weakly continuous in time, and
\begin{equation}\label{eq:localmoments}
 \mi a\dd\mu_t(a)\le A_0,\qquad
 \mi a^{-1}\dd\mu_t(a)\le\frac{m_0}{1-m_0t},\qquad
 0\le\theta(t,x)\le f(x).
\end{equation}
\end{proposition}

Let us start with a preliminary technical lemma:

\begin{lemma} 
\label{lem:mixture-closure}
Let $(\lambda_j)_{j\ge1}$ be Borel probability measures on
$(0,\infty)$, and set
\[
 g_j(x):=\int_{(0,\infty)}p_a(x)\,\dd\lambda_j(a).
\]
Assume that, for some finite constants $A_*,M_*>0$,
\[
 \int_{(0,\infty)}a\,\dd\lambda_j(a)\le A_*,
 \qquad
 \int_{(0,\infty)}a^{-1}\,\dd\lambda_j(a)\le M_*, \quad \textup{for all } j \geq 1.
\]
If $g_j(x)\to g(x)$ for every $x\in\mathbb R$, for some function $g: \R \to \R$, then there
exists a unique Borel probability measure $\lambda$ on
$(0,\infty)$ such that
\[
 g(x)=\int_{(0,\infty)}p_a(x)\,\dd\lambda(a)
 \qquad(x\in\mathbb R).
\]
Moreover,
\[
 \int_{(0,\infty)}a\,\dd\lambda(a)\le A_*,
 \qquad
 \int_{(0,\infty)}a^{-1}\,\dd\lambda(a)\le M_*,
\]
and the entire sequence $(\la_j)_j$ converges weakly to $\lambda$. 
\end{lemma}

\begin{proof}
For $0<\varepsilon<R$, the moment bounds give
\[
 \lambda_j\bigl((0,\infty)\setminus[\varepsilon,R]\bigr)
 \le \varepsilon M_*+\frac{A_*}{R}.
\]
Thus $(\lambda_j)_j$ is tight on $(0,\infty)$, and Prokhorov's
theorem \cite[Theorem~5.1, p.~59]{Billingsley} gives a weakly
convergent subsequence with limit a Borel probability measure
$\lambda$ on $(0,\infty)$.

Since $a\mapsto p_a(x)$ is bounded and continuous for each
$x\in\mathbb R$, passage to the limit gives
\[
 g(x)=\int_{(0,\infty)}p_a(x)\,\dd\lambda(a).
\]
The uniqueness of $\lambda$ follows by arguing as in the proof of
Lemma~\ref{lem:stieltjes}. Hence, every weakly convergent subsequence has the same limit. Tightness then implies
$\lambda_j\rightharpoonup\lambda$ for the entire sequence.

Finally, weak convergence applied to the bounded continuous
functions $a\mapsto \min\{ a, n\}$ and $a\mapsto \min\{a^{-1},n\}$,
followed by monotone convergence as $n\to\infty$, yields
\[
 \int_{(0,\infty)}a\,\dd\lambda(a)\le A_*,
 \qquad
 \int_{(0,\infty)}a^{-1}\,\dd\lambda(a)\le M_*.
\]
\end{proof}

\begin{proof}[Proof of Proposition \ref{prop:local}]
Let $\mu_0$ denote the probability measure representing $f$.
Its uniqueness follows from Lemma~\ref{lem:stieltjes}, as explained
in Lemma \ref{lem:mixture-closure}. Since $a^{-1}>0$ on $(0,\infty)$ and $\mu_0$ has
mass one, $0<m_0<\infty$. Fix an arbitrary
$0<\tau<m_0^{-1}$ and put
\begin{equation}\label{eq:local-envelope-definitions}
 M(t):=\frac{m_0}{1-m_0t},\qquad
 \Gamma(t):=\int_0^t M(s)\dd s=-\log(1-m_0t),
 \qquad 0\le t\le\tau.
\end{equation}
In particular, $M$ is increasing, $\Gamma(\tau)<\infty$, and
$m_0e^{\Gamma(t)}=M(t)$. We construct the solution on this fixed interval
and keep the same interval for every Sobolev order.
 
For the sake of clarity, we split the rest of the proof into several steps.

\noindent
\textbf{Step 1.} \emph{Construction of each smooth iterate and its probability measure representation.}
Set $\theta^{(0)}(t,x)=f(x)$ and $\mu_t^{(0)}=\mu_0$.
We construct $\theta^{(n+1)}$ by solving the equation
\begin{equation}\label{eq:Picard}
 \partial_t\theta^{(n+1)}=u^{(n)}\partial_x\theta^{(n+1)},
 \qquad u^{(n)}=H\theta^{(n)},\qquad \theta^{(n+1)}(0,\cdot)=f,
\end{equation}
and we set
\[
 A_n(t):=\int_{(0,\infty)} a\dd\mu_t^{(n)}(a) \quad \textup{and} \quad 
 m_n(t):=\int_{(0,\infty)} a^{-1}\dd\mu_t^{(n)}(a).
\]
The induction hypotheses are then that $\theta^{(n)}$ belongs to
$C([0,\tau];H^k)$ for every integer $k\ge0$, that it has a
weakly continuous probability measure representation $\mu_t^{(n)}$, and
that its moments satisfy $A_n(t)\le A_0$, $m_n(t)\le M(t)$.
Note that these hypotheses hold for $n=0$.
 
The method of characteristics gives existence and uniqueness, and the required Sobolev regularity and time continuity follow from standard linear transport estimates.
 
The induction hypotheses and Lemma~\ref{lem:mixture} imply
\[
 u^{(n)}(t,x)=xB_n(t,x^2) \quad \textup{with} \quad
 B_n(t,r)=\int_{(0,\infty)}\frac a{a^2+r}\dd\mu_t^{(n)}(a).
\]
In particular $u^{(n)}(t, \cdot)$ is odd, and so are the associated characteristics in space by uniqueness. Since $f$ is even, it then    follows that $\theta^{(n+1)}(t,\cdot)$ is even as well. Moreover, using the fact that the associated characteristics are odd in space, we get that $\theta^{(n+1)}(t,0)=f(0)=1$.

For $r>0$, put $Q_{n+1}(t,r)=\theta^{(n+1)}(t,\sqrt r)$.
The chain rule in \eqref{eq:Picard} gives
\[
 \partial_tQ_{n+1}(t,r)=2rB_n(t,r)\partial_rQ_{n+1}(t,r),
 \qquad Q_{n+1}(0,r)=\int_{(0,\infty)}\frac{a^2}{a^2+r}\dd\mu_0(a).
\]
This is exactly the equation in Lemma~\ref{lem:frozen}. Its
coefficient measures are weakly continuous by the induction
hypotheses, and therefore weakly measurable; their moments satisfy
$A_n(t)\le A_0$, $m_n(t)\le M(t)$ on $[0,\tau]$, so both are integrable
in time. Thus, all hypotheses of Lemma \ref{lem:frozen} are satisfied.
Its characteristic solution agrees with $Q_{n+1}$, by the
uniqueness established there. Consequently $\theta^{(n+1)}(t,\cdot)$
has a unique probability measure representation $\mu_t^{(n+1)}$ for
$x>0$, for $x<0$ by evenness, and for $x=0$ by normalization.
The moment bounds in the same lemma give
\[
 A_{n+1}(t)\le A_0 \quad \textup{and} \quad
 m_{n+1}(t)\le m_0\exp\left(\int_0^t m_n(s)\dd s\right)
             \le m_0e^{\Gamma(t)}=M(t).
\]
Finally, using \eqref{eq:ftcomparison}, we get that
$$
0 \leq \te^{(n+1)}(t,x) \leq f(x), \quad \textup{for } (t,x) \in [0,\tau] \times \RR.
$$
The moment bounds are uniform in time, and by construction we have pointwise
time continuity of $\theta^{(n+1)}$. Then, using Lemma \ref{lem:mixture-closure} we infer weak continuity of $\mu_t^{(n+1)}$. This
completes the induction and justifies the next use of Lemma \ref{lem:frozen}. In particular, no measurability or continuity of an
as-yet-unconstructed measure family was assumed.

In short, for all $n\ge0$,  we have obtained  $\theta^{(n)}$ satisfying
\begin{equation}\label{eq:envelope}
 A_n(t)\le A_0,\qquad m_n(t)\le M(t),\qquad
 0\le\theta^{(n)}(t,x)\le f(x).
\end{equation}
Moreover, $\te^{(n)} \in C([0,\tau]; H^k)$ for every $k \geq 0$ and it has a
weakly continuous probability measure representation $(\mu_t^{(n)})_{0 \leq t \leq \tau}$.

\noindent \textbf{Step 2.} \emph{Convergence in $C([0,\tau]; L^2)$ } First, for an arbitrary but fixed integer $k\ge2$, let 
$$
E_n(t)=\norm{\theta^{(n)}(t,\cdot)}_{H^k}.
$$
Combining Lemmas \ref{lem:mixture} and \ref{lem:energy} with \eqref{eq:envelope}, it follows that
\[
 E_{n+1}'(t)\le C_kM(t)(E_{n+1}(t)+E_n(t)),
 \qquad E_{n+1}(0)=\norm f_{H^k}.
\]
Setting $a_k(t) = C_k \Gamma(t)$ and $F_k = \|f\|_{H^k}$, and using induction, one gets that
\begin{equation}\label{eq:local-uniform-sobolev}
 E_n(t)\le F_ke^{2a_k(t)}, \quad \textup{for all } n \geq 0, \textup{ and all } t \in [0,\tau].
\end{equation}

Next, we set $d^{(n)}:=\theta^{(n+1)}-\theta^{(n)}$ for all $n\ge0$.
For all $n\ge1$, subtracting consecutive equations in \eqref{eq:Picard}, we get that
\begin{equation}\label{eq:difference}
 d^{(n)}_t=u^{(n)}d^{(n)}_x+(Hd^{(n-1)})\theta_x^{(n)},
 \qquad d^{(n)}(0,\cdot)=0.
\end{equation}
Pairing with $d^{(n)}$ in $L^2$
and integrating by parts gives
\begin{align*}
 \frac12\frac{\dd}{\dd t}\norm{d^{(n)}}_{L^2}^2
 &=-\frac12\int_\R u_x^{(n)}(d^{(n)})^2\dd x
   +\int_\R d^{(n)}(Hd^{(n-1)})\theta_x^{(n)}\dd x\\
 &\le\frac12M(t)\norm{d^{(n)}}_{L^2}^2
       +M(t)\norm{d^{(n)}}_{L^2}\norm{d^{(n-1)}}_{L^2}.
\end{align*}
Let $D_n(t)=\norm{d^{(n)}(t,\cdot)}_{L^2}$. Dividing the preceding inequality by $D_n$ where $D_n>0$, and using $\sqrt{D_n^2+\varepsilon^2}$ followed by $\varepsilon\downarrow0$ at its zeros, yields 
\[
 D_n'(t)\le M(t)(D_n(t)+D_{n-1}(t)),\qquad D_n(0)=0.
\]
Thus, taking into account the definition of $\Gamma(t)$ (see \eqref{eq:local-envelope-definitions}), we get that
\[
 Z_n(t):=e^{-\Gamma(t)}D_n(t)
 \le\int_0^t M(s)Z_{n-1}(s)\dd s, \quad \textup{for all } n \geq 1.
\]
The quantity $D_*:=\sup_{0\le t\le\tau}Z_0(t)$ is finite by \eqref{eq:local-uniform-sobolev}. Then, by induction
\begin{equation}\label{eq:local-factorial-convergence}
 Z_n(t)\le D_*\frac{\Gamma(t)^n}{n!},\qquad
 D_n(t)\le D_*e^{\Gamma(t)}\frac{\Gamma(t)^n}{n!}, \quad \textup{for all } n \geq 0.
\end{equation}
This implies that $\sum_{n\ge0}\sup_{t\le\tau}D_n(t)$ is finite.
For $q>p$, the identity
$\theta^{(q)}-\theta^{(p)}=\sum_{n=p}^{q-1}d^{(n)}$ therefore
shows that $\theta^{(n)}$ is Cauchy in $C([0,\tau];L^2)$.
Completeness gives a limit $\theta \in C([0,\tau];L^2)$ .

\noindent \textbf{Step 3.} \emph{Convergence in all Sobolev spaces and uniqueness.} This is standard, so we just give a sketch. Interpolation with the higher order uniform bounds implicit in the proof of \eqref{eq:local-uniform-sobolev} gives convergence in $C([0,\tau]; H^k).$ Moreover, a standard bootstrap argument gives smoothness. Finally, a classical application of Gr\"onwall's inequality shows uniqueness among all smooth Sobolev solutions. 

\noindent \textbf{Step 4.} \emph{Conclusion}. Fix $t\in[0,\tau]$. By Step 3, $\theta^{(n)}(t,\cdot)$ converges
to $\theta(t,\cdot)$ uniformly. Moreover, the probability measures
$\mu_t^{(n)}$ have the moment bounds $A_n(t)\le A_0$ and
$m_n(t)\le M(t)$. Lemma \ref{lem:mixture-closure}, applied with $A_*=A_0$ and $M_*=M(t)$,
therefore gives a unique probability measure $\mu_t$ such that
\[
 \mu_t^{(n)}\longrightarrow\mu_t\quad\hbox{weakly},\quad
 \theta(t,x)=\int_{(0,\infty)} p_a(x)\dd\mu_t(a),\quad
 \int_{(0,\infty)} a\dd\mu_t(a)\le A_0,\quad
 \int_{(0,\infty)} a^{-1}\dd\mu_t(a)\le M(t).
\]
Passing to the pointwise limit in \eqref{eq:envelope} also gives
$0\le\theta(t,x)\le f(x)$. These statements prove in particular all the
inequalities in \eqref{eq:localmoments}.

The first moment is finite, so Lemma~\ref{lem:mixture} applies
to this representation and yields
\[
 u(t,x)=H\theta(t,x)=\int_{(0,\infty)} h_a(x)\dd\mu_t(a).
\]
The equality holds for the continuous representatives, and is
consistent with the Sobolev field $u=H\theta$ constructed in Step 3.
This proves both identities in \eqref{eq:nonlinearrep}. The same
lemma also identifies the inverse moment with $u_x(t,0)$. Let us explicitly point out that  no differentiation of the measure in time is required.

Finally, $\theta\in C([0,\tau];H^1)$ gives pointwise time
continuity, and the measures $\mu_t$ have uniform moment bounds
$A_0$ and $M(\tau)$. The time-continuity conclusion inferred from Lemma \ref{lem:mixture-closure}
therefore proves their weak continuity on $[0,\tau]$, including
at $t=0$, where uniqueness gives $\mu_{t=0}=\mu_0$.
Since $\tau<m_0^{-1}$ was arbitrary, we have proved the proposition
on every asserted interval. Solutions constructed with different
choices of $\tau$ agree on their common interval by uniqueness.
\end{proof}

\begin{corollary}\label{cor:invariance}
For the datum $\theta_0(x)= (1+x^2)^{-1}$, representation \eqref{eq:nonlinearrep} holds
at every $0 \leq t<T$. The representing probabilities satisfy
\begin{equation}\label{eq:globalmoments}
 A(t):=\mi a\dd\mu_t(a)=\frac1\pi\norm{\theta(t,\cdot)}_{L^1}\le1,
 \qquad
 m(t):=m_{\mu_t}=\mi a^{-1}\dd\mu_t(a)=u_x(t,0).
\end{equation}
Moreover, 
\begin{equation}\label{eq:amplitude}
 0\le\theta(t,x)\le\frac1{1+x^2},\qquad
 \theta(t,0)=\norm{\theta(t,\cdot)}_{L^\infty}=1,\qquad
 \norm{u(t,\cdot)}_{L^\infty}\le\frac12,
\end{equation}
 and
\begin{equation}\label{eq:globalderivatives}
 \norm{u_x(t,\cdot)}_{L^\infty}=m(t),\qquad
 \norm{\theta_x(t,\cdot)}_{L^\infty} \le m(t).
\end{equation}
\end{corollary}

\begin{proof}
The initial measure is $\mu_0 = \delta_1$. Restart Proposition~\ref{prop:local}
at each time when the solution is smooth and has this representation.
On every compact smooth-time interval,
$m(t)=u_x(t,0)$ is bounded by Sobolev embedding and
Remark~\ref{rem:moment}; the allowed restart lengths therefore have a
positive lower bound. Finitely many restarts cover the interval, and
uniqueness identifies the resulting solutions. Each restart decreases
the scalar pointwise and preserves $A(t)\le1$. The other assertions follow
from Lemma~\ref{lem:mixture} and the probability normalization.
\end{proof}

\section{A kernel identity for the forcing}\label{sec:forcing}
For $t \in [0,T)$, let $(\mu_t)_t$ be the family of probability measures in
\eqref{eq:nonlinearrep}. Its first and inverse first moments are finite
by Corollary~\ref{cor:invariance}. Likewise, set
\[
 B(t,r) :=\int_{(0,\infty)}\frac{a}{a^2+r}\dd\mu_t(a),
 \quad r\ge0.
\]
Then, \eqref{eq:nonlinearrep} reads $u(t,x)=xB(t,x^2)$. In particular, it follows that $B(t,0)=m(t)$ and
$u_x(t,0)=m(t)$.

We first derive the evolution equation used in this section. Since the
solution is smooth in every Sobolev space, $H$ commutes with
$\partial_t$ and $\partial_x$. Thus, using that $u=H\theta$ and $H^2=-I$, we get that
\[
 Hu_x=H^2\theta_x=-\theta_x.
\]
Consequently, the scalar equation in \eqref{eq:CCF} can be written
$\theta_t=-uHu_x$. Applying then $H$ gives
$u_t=-H(uHu_x)$, and subtracting $uu_x$ from both sides yields
\begin{equation}\label{eq:burgers}
 u_t-uu_x=-\mathcal N(u),\qquad
 \mathcal N(u):=H(uHu_x)+uu_x=-H(u\theta_x)+uu_x,
\end{equation}
which is in turn equivalent to
\[
 Du=-\mathcal N(u),
\]
where we recall that $D := \partial_t-u\partial_x$.

\begin{lemma}\label{lem:force}
For $\al > 0$, set ${\rm D }_\alpha(x) := \alpha^2 + x^2$. Then, it follows that
\begin{equation}\label{eq:forcekernel}
 \frac{\mathcal N(u)(t,x)}x
 =\iint_{(0,\infty)^2}\frac{ab({\rm D }_a(x)+{\rm D }_b(x))}{2(a+b)^2 {\rm D }_a(x){\rm D }_b(x)}\, \dd \mu_t(a)\,\dd \mu_t(b)
 \ge\frac{u(t,x)^2}{4x^2}, \quad \textup{for } x > 0 \textup{ and } t \in [0,T).
\end{equation}
In particular, 
\begin{equation}\label{eq:central-force}
 \mathcal N(u)_x(t,0)=\frac12m(t)^2-I(t) \quad \textup{with} \quad
 I(t):=\iint_{(0,\infty)^2}\frac{\dd\mu_t(a)\,\dd\mu_t(b)}{(a+b)^2}.
\end{equation}
\end{lemma}
\begin{proof}
We suppress the fixed time and write $\mu=\mu_t$ and $B(r)=B(t,r)$. Thus
\[
 \theta(x)=\int_{(0,\infty)} p_b(x)\dd\mu(b),\qquad
 u(x)=\int_{(0,\infty)} h_a(x)\dd\mu(a),
\]
and the two moments
\[
 A=\int_{(0,\infty)} a\dd\mu(a),\qquad m=\int_{(0,\infty)} a^{-1}\dd\mu(a)
\]
are finite. 

For $a,b>0$, we set
\[
 q_{a,b}(y):=\frac{y^2}{(y^2+a^2)(y^2+b^2)}.
\]
If $a\ne b$, it follows that
\begin{equation}\label{eq:forcing-partial-fractions}
 q_{a,b}(y)
 =-\frac{a^2}{b^2-a^2}\frac1{y^2+a^2}
  +\frac{b^2}{b^2-a^2}\frac1{y^2+b^2} = \frac{1}{b^2-a^2} (p_b(y) - p_a(y)),
\end{equation}
and so, since $Hp_a=h_a$, that
\begin{align*}
 Hq_{a,b}(x) 
 & = \frac{1}{b^2-a^2} (h_b(x)-h_a(x))=\frac{x(x^2-ab)}{(a+b)(a^2+x^2)(b^2+x^2)}.
\end{align*}
Setting 
\[
r=x^2,\qquad {\rm D}_a=a^2+r,\qquad {\rm D}_b=b^2+r,\qquad S=a+b,
\]
we have therefore proved that
\begin{equation}\label{eq:forcing-J}
 J_{a,b}(r):=\frac1xHq_{a,b}(x)
 =\frac{r-ab}{S {\rm D}_a{\rm D}_b} \quad \textup{for all }a\ne b,\ x\ne0.
\end{equation}
The case $a=b$ follows by continuity without any appeal to a formal cancellation.
Fix $a>0$ and let $b\to a$. Then $q_{a,b}\to q_{a,a}$ in
$L^2$ by dominated convergence, so the $L^2$ boundedness of $H$ gives
$Hq_{a,b}\to Hq_{a,a}$ in $L^2$. The rational expression on the
right side of the formula for $Hq_{a,b}$ also converges in $L^2$ to
the expression obtained by setting $b=a$. Uniqueness of the $L^2$ limit proves the identity for
$a=b$ almost everywhere in $x$. Since both sides have continuous
representatives the identity holds for every
$x$. Hence, \eqref{eq:forcing-J} is valid for all $a,b>0$ and
$x\ne0$. Moreover, the formula for $Hq_{a,b}$ also extends to $x=0$ after
multiplication by $x$.

We next relate this transform to the product in
$-H(u\theta_x)$. By direct calculus, 
\begin{equation}\label{eq:forcing-parameter-derivative}
 h_a(y)p_b'(y)=ab\,\partial_bq_{a,b}(y).
\end{equation}
Also, on any compact parameter square
$[\varepsilon,R]^2\subset(0,\infty)^2$, the map
$(a,b)\mapsto q_{a,b}$ is continuously differentiable with values
in $L^2$. Since $H$ is bounded
on $L^2$, parameter differentiation commutes with $H$ there. Thus,
\begin{equation}\label{eq:forcing-H-product}
 \frac1xH(h_ap_b')(x)
 =\frac{ab}{x}\partial_bHq_{a,b}(x)
 =ab\,\partial_bJ_{a,b}(r), \quad \textup{for } (a,b) \in [\ep,R]^2,\ x \neq 0.
\end{equation}
This identity is obtained first in $L^2$, but its two sides have
continuous representatives, so it holds pointwise.
 
For the moment, we restrict $a$ and $b$ to
$E:=[\varepsilon,R]$, and define
\[
 \theta_E(x):=\int_Ep_b(x)\dd\mu(b),\qquad
 u_E(x):=\int_Eh_a(x)\dd\mu(a),
\]
and
\[
 \mathcal N_E(x):=-H(u_E\theta_{E,x})(x)+u_E(x)u_{E,x}(x).
\]
Here we are using the shortened notation $\pd_x \te_E = \te_{E,x}$ and $\pd_x u_E = u_{E,x}$. 
All parameter integrals on $E$ can be interchanged with products,
derivatives and $H$. Therefore
\[
 \mathcal N_E
 =\iint_{E^2}\bigl[-H(h_ap_b')+h_ah_b'\bigr]
          \dd\mu(a)\dd\mu(b).
\]
This is an equality in $L^2$, but both sides have continuous
representatives, so all subsequent pointwise evaluations are
legitimate. Also, the product measure $\mu\otimes\mu$ is invariant under exchanging
$a$ and $b$. Hence,
\begin{equation}\label{eq:forcing-symmetrization}
 \mathcal N_E
 =\frac12\iint_{E^2}
   \bigl[-H(h_ap_b')-H(h_bp_a')+(h_ah_b)'\bigr]
   \dd\mu(a)\dd\mu(b).
\end{equation}

Combining \eqref{eq:forcing-H-product} with its version after
exchanging $a$ and $b$, we see from
\eqref{eq:forcing-symmetrization} that the symmetrized kernel of
$\mathcal N_E(x)/x$ is
\begin{equation}\label{eq:forcing-K-definition}
 K_{a,b}(r):=\frac{(h_ah_b)'(x)}{2x}
 -\frac{ab}{2}(\partial_a+\partial_b)J_{a,b}(r).
\end{equation}
On the other hand, by direct calculations
\begin{align*}
 K_{a,b}(r)
 &=\frac{ab}{2{\rm D}_a{\rm D}_b}
   +\frac{ab(r-ab)}{S^2{\rm D}_a{\rm D}_b}=\frac{ab\{S^2+2(r-ab)\}}{2S^2{\rm D}_a{\rm D}_b}=\frac{ab(a^2+b^2+2r)}{2S^2{\rm D}_a{\rm D}_b}
  =\frac{ab({\rm D}_a+{\rm D}_b)}{2(a+b)^2{\rm D}_a{\rm D}_b} \geq 0.
\end{align*}
We have thus proved that, for every truncated parameter interval $E$ and all
$x>0$,
\begin{equation}\label{eq:forcing-truncated}
 \frac{\mathcal N_E(x)}x
 =\iint_{E^2}K_{a,b}(x^2)\dd\mu(a)\dd\mu(b).
\end{equation}

Next, choose $E_j := [j^{-1},j]$. Since $\int_{(0,\infty)}\sqrt a\,\dd\mu(a)<\infty$, the kernel norm in \eqref{eq:Hkernels} gives $u_{E_j}\to u$ in $L^2$. Also, the bounds $|p_a'|,|h_a'|\le a^{-1}$ give
\[
 \norm{\theta_{E_j,x}-\theta_x}_{L^\infty}
 +\norm{u_{E_j,x}-u_x}_{L^\infty}
 \le 2\int_{(0,\infty)\setminus E_j}a^{-1}\dd\mu(a)\longrightarrow0.
\]
Consequently both $u_{E_j}\theta_{E_j,x}\to u\theta_x$ and $u_{E_j}u_{E_j,x}\to uu_x$ in $L^2$. The $L^2$ boundedness of $H$ therefore implies
\begin{equation}\label{eq:forcing-N-convergence}
\cN_{E_j} \xrightarrow{j \to \infty} \cN(u) \quad \textup{in } L^2. 
\end{equation}
On the other hand, it follows that
\begin{equation} \label{eq:forcing-kernel-majorant}
0 \leq K_{a,b}(r) \leq \frac{1}{ab}, \quad \textup{for all } a, b > 0, \ r \geq 0\,. 
\end{equation}
Moreover,
$$
\iint_{(0,\infty)^2} \frac{1}{ab} \, \dd \mu(a) \dd \mu(b) = m^2 < \infty\,.
$$
Setting
\[
 F(x):=x\iint_{(0,\infty)^2} K_{a,b}(x^2)\dd\mu(a)\dd\mu(b),
 \quad x\ge0,
\]
the bound \eqref{eq:forcing-kernel-majorant} and \eqref{eq:forcing-truncated} give, for every $R>0$,
\[
 \sup_{0\le x\le R}|\mathcal N_{E_j}(x)-F(x)|
 \le R\iint_{(0,\infty)^2\setminus E_j^2}\frac{\dd\mu(a)\dd\mu(b)}{ab}
 \longrightarrow0.
\]
Here the convergence follows from the integrability of $(ab)^{-1}$, and the identity at $x=0$ holds by continuity. Thus,
$$
\cN_{E_j} \xrightarrow{j \to \infty} F \quad \textup{in } L_{\rm loc}^{\infty}([0,\infty))\,.
$$
Furthermore, the same upper bound for $K_{a,b}$ and dominated convergence show that $F$ is continuous on $[0,\infty)$. 

By \eqref{eq:forcing-N-convergence}, a subsequence of
$\mathcal N_{E_j}$ converges almost everywhere on $(0,\infty)$ to
$\mathcal N(u)$. By the locally uniform convergence just proved, the
same subsequence converges everywhere on $(0,\infty)$ to $F$. Hence
$\mathcal N(u)=F$ almost everywhere on $(0,\infty)$. Both sides are
continuous there, so equality holds at every $x>0$, giving
\begin{equation}\label{eq:forcing-full-kernel}
 \frac{\mathcal N(u)(x)}x
 =\iint_{(0,\infty)^2} K_{a,b}(x^2)\dd\mu(a)\dd\mu(b),
\end{equation}
which is precisely the first identity in \eqref{eq:forcekernel}. In order to obtain the lower bound in \eqref{eq:forcekernel} we just use that
\[
 K_{a,b}(r)
 \ge\frac{ab}{4 {\rm D}_a{\rm D}_b}.
\]
Indeed, plugging this inequality into \eqref{eq:forcing-full-kernel} gives
\begin{align*}
 \frac{\mathcal N(u)(x)}x
 &\ge\frac14\iint_{(0,\infty)^2}
       \frac{ab}{(a^2+x^2)(b^2+x^2)}\dd\mu(a)\dd\mu(b)=\frac14\left(\int_{(0,\infty)}\frac{a}{a^2+x^2}\dd\mu(a)\right)^2
  =\frac14B(x^2)^2
  =\frac{u(x)^2}{4x^2}.
\end{align*}

Finally, at $r=0$, the explicit kernel is
\begin{align*}
 K_{a,b}(0)
 &=\frac{ab(a^2+b^2)}{2(a+b)^2a^2b^2}
  =\frac{a^2+b^2}{2ab(a+b)^2} =\frac1{2ab}-\frac1{(a+b)^2}.
\end{align*}
Both terms on the right hand side are integrable. Thus, combining the dominated convergence theorem with \eqref{eq:forcing-kernel-majorant} yields
\begin{align*}
 \lim_{x\downarrow0}\frac{\mathcal N(u)(x)}x=\iint_{(0,\infty)^2} K_{a,b}(0)\dd\mu(a)\dd\mu(b)=\frac12\left(\int_{(0,\infty)} a^{-1}\dd\mu(a)\right)^2
   -\iint_{(0,\infty)^2}\frac{\dd\mu(a)\dd\mu(b)}{(a+b)^2}=\frac12m^2-I.
\end{align*}
This also shows that $F(x)=\mathcal N(u)(x)$ tends to zero
as $x\downarrow0$. Hence, $\mathcal N(u)(0)=0$ by continuity, and
smoothness implies that the preceding right-hand limit is precisely
$\mathcal N(u)_x(0)$. This proves \eqref{eq:central-force} and
completes the proof.
\end{proof}

\section{Proof of the main theorem: quantitative statement and estimates}
\label{sec:main-proof}

We now combine the probability measure representation of Section~\ref{sec:invariance}
with the forcing identity of Section~\ref{sec:forcing} to prove Theorem \ref{thm:intro}. We actually prove a quantitative version of it, which reads as follows.

\begin{theorem}
\label{thm:main}
Consider \eqref{eq:CCF} with initial data
\begin{equation}\label{eq:data}
\theta_0(x)=\frac1{1+x^2}.
\end{equation}
Then, there is a unique maximal smooth Sobolev solution to \eqref{eq:CCF} on $[0,T)$,
and, for every non-negative integer $k$, it follows that
$
 \theta,u\in C([0,T);H^k).
$
Moreover, for every $0 \leq t<T$, the
solution has the unique probability measure representation
\eqref{eq:nonlinearrep}, with $\mu_0=\delta_1$; the measures $(\mu_t)_{0 \leq t < T}$
are weakly continuous in time and have finite first and inverse
first moments; and the following conclusions hold:
\begin{enumerate}
\item \label{itemone} $
 0\le\theta(t,x)\le (1+x^2)^{-1},\ 
 \theta(t,0)=\norm{\theta(t, \cdot)}_{L^\infty}=1$ and $
 \norm{u(t,\cdot)}_{L^\infty}\le\frac12.
$ \smallbreak
\item \label{itemtwo} The maximal lifespan satisfies
$
4/3\le T\le2,
$
and the positive quantity
$
 m(t) = u_x(t,0)=\Lambda\theta(t,0)
$
is strictly increasing, starts at $m(0)=1$, and tends to infinity
as $t\uparrow T$. More precisely,
\begin{equation*} 
 m(t)=\norm{u_x(t,\cdot)}_{L^\infty} \quad \textup{and} \quad
 \frac4{3(T-t)}\le m(t)\le\frac2{T-t}.
\end{equation*}
\item \label{itemthree} The solution remains uniformly bounded in $L^{\infty}([0,T); C^{2/5})$. In fact, it follows that 
$$
\max\{[\theta(t,\cdot)]_{C^{2/5}},\ [u(t,\cdot)]_{C^{2/5}}\}
 \le \frac{5^{8/5}}{2}, \quad \textup{\textit{for all} } t \in [0,T).
$$

\item \label{itemfour}
Setting $\delta:=T-t$, it follows that
\begin{equation*}
 \frac{\pi}{12\delta\log(3/\delta)}
 \le\norm{\theta_x(t,\cdot)}_{L^\infty}\le\frac2\delta, \quad \textup{for all } 0 < \delta \leq \frac12.
\end{equation*}
\item \label{itemfive}  There exists $c > 0$ such that
\begin{equation*}
 \min\{[\theta(t,\cdot)]_{C^{2/3}},[u(t,\cdot)]_{C^{2/3}}\}
 \ge c \left(\frac{\log(4m(t))}{\log4}\right)^{1/9}.
\end{equation*}
In particular, it follows that
$$
\lim_{t \uparrow T} \|\te(t,\cdot)\|_{C^{2/3}} = \lim_{t \uparrow T} \|u(t,\cdot)\|_{C^{2/3}} = \infty.
$$
\end{enumerate}
\end{theorem}

The rest of this section is devoted to the proof of this result. For simplicity, we divide the proof into several stages.

\begin{proof}[\textbf{Proof of Theorem \ref{thm:main} -- Existence, uniqueness and \ref{itemone}}] The initial active scalar $\theta_0$ is precisely $p_1$, represented by the probability measure
$\delta_1$. It belongs to every $H^k$ and its first and inverse first moments both equal one.
Proposition~\ref{prop:local} therefore constructs a smooth solution
on each $[0,\tau]$ with $\tau<1$, and proves uniqueness in the full
smooth Sobolev class. We then set
\[
 T:=\sup\{\tau>0:\ \text{a smooth Sobolev solution with
 data \eqref{eq:data} exists on }[0,\tau]\}.
\]
Solutions on any two such intervals agree on their intersection
by uniqueness. They consequently define a single solution on
$[0,T)$, continuous into every $H^k$ on every compact time
interval, and hence on $[0,T)$. This is the maximal solution
we were looking for.

Corollary~\ref{cor:invariance} supplies the representing probability measures
and their finite moments on the full lifespan. The
pointwise decrease in Proposition~\ref{prop:local}, applied on
the successive intervals in that corollary, gives
$0\le\theta(t,x)\le p_1(x)$. Thus $\norm{\theta(t,\cdot)}_{L^\infty}=1$.
Furthermore,
\[
 |h_a(x)|=\frac{a|x|}{a^2+x^2}\le\frac12.
\]
Integrating against a probability yields
$|u(t,x)|\le1/2$. This proves
$(i)$. Uniqueness and weak continuity of $(\mu_t)_t$ follow locally from
Proposition~\ref{prop:local}; the local conclusions agree on
overlaps and therefore hold throughout $[0,T)$.
\end{proof}

\begin{remark}
Applying Lemma~\ref{lem:mixture} to $f=\theta(t,\cdot)$ and
$\mu=\mu_t$, with $Hf=u(t,\cdot)$, identifies the inverse first
moment $m(t)=m_{\mu_t}$ with $u_x(t,0)$ and gives
\begin{equation}\label{eq:proof-mixture-derivatives}
 m(t)=u_x(t,0)=\norm{u_x(t,\cdot)}_{L^\infty}>0,\qquad
 \norm{\theta_x(t,\cdot)}_{L^\infty}\le m(t).
\end{equation}
\end{remark}

Next, we derive an equation for $m'(t)$. This equation immediately allows us to prove Theorem \ref{thm:main} $(ii)$. Nevertheless, it will also be crucial for the subsequent H\"older estimates so we provide it as an independent result. Note that all derivatives in the calculations below are classical on
compact time intervals. In particular, the equation and
the available Sobolev regularity justify evaluation and spatial
differentiation at the origin.

\begin{lemma}
\label{lem:lifespan}
For the maximal solution just constructed, let
\[
 I(t):=\iint_{(0,\infty)^2}
             \frac{\dd\mu_t(a)\dd\mu_t(b)}{(a+b)^2}.
\]
Then
\begin{equation}\label{eq:mODE}
 m'(t)=\frac12m(t)^2+I(t),\qquad m(0)=1,
\end{equation}
and
\begin{equation}\label{eq:Riccati}
 \frac12m(t)^2\le m'(t)\le\frac34m(t)^2.
\end{equation}
\end{lemma}

\begin{proof}
First of all, note that
\[
 0<I(t)\le\frac14\iint_{(0,\infty)^2}\frac{\dd\mu_t(a)\dd\mu_t(b)}{ab}
          =\frac14m(t)^2.
\]
On the other hand, differentiating \eqref{eq:burgers} in $x$, we get that
\[
 u_{xt}-u_x^2-uu_{xx}=-\mathcal N(u)_x.
\]
At $x=0$, oddness gives $u(t,0)=0$ and $u_x(t,0)=m(t)$.
Lemma~\ref{lem:force} therefore yields
\[
 m'(t)=m(t)^2-\mathcal N(u)_x(t,0)
    =m(t)^2-\left(\frac12m(t)^2-I(t)\right)
    =\frac12m(t)^2+I(t).
\]
This proves \eqref{eq:mODE}. Combining this identity with the initial bounds on $I(t)$ proves
\eqref{eq:Riccati}.
\end{proof}

\begin{remark}
    The lower bound in \eqref{eq:Riccati} was already obtained in the Ph.D. thesis of the first author (see \cite[(4.1) and Lemma 4.2]{CC2010}). This is enough to prove finite time blow-up of the $C^1$ norm. Nevertheless, it does not suffice to estimate the H\"older regularity up to the singular time.
\end{remark}

\begin{proof}[\textbf{Proof of Theorem \ref{thm:main} -- \ref{itemtwo}}]
First of all, \eqref{eq:data} and \eqref{eq:Hkernels} give $u(0,x)=x/(1+x^2)$, hence $m(0)=1$. Combining this with \eqref{eq:Riccati} shows that $m'(t) > 0$, and so that $m$ is strictly increasing on $(0,T)$. 

Since $m>0$, dividing \eqref{eq:Riccati} by $-m^2$ reverses the
inequalities and gives
\begin{equation}\label{eq:reciprocal}
 -\frac34\le\left(\frac1m\right)'\le-\frac12.
\end{equation}
Integration from $0$ to any $t<T$ then gives
\begin{equation}\label{eq:reciprocal-initial}
 1-\frac34t\le\frac1{m(t)}\le1-\frac12t.
\end{equation}
If $T>2$, the right inequality at time $t=2$
would give $0 < 1/m(2)\le0$, which is impossible. Thus, we conclude that $T\le2$.

We now prove that $m(t) \to \infty$ as $t \uparrow T$. We argue by contradiction and, taking into account that $m$ is strictly increasing, assume that there exists $M > 0$ such that $m(t)\le M<\infty$ for all $t<T$. Then,  we choose
$t_0<T$ so close to $T$ that $T-t_0<1/(2M)$.
The datum $\theta(t_0)$ admits this probability measure representation, belongs to every
$H^k$, and has finite first moment and inverse first moment
$m(t_0)\le M$. Proposition~\ref{prop:local}, with time translated
by $t_0$, constructs a solution on
$[t_0,t_0+1/(2M)]$, since
$1/(2M)<1/m(t_0)$. Uniqueness identifies it with the original
solution on the overlap. As $t_0+1/(2M)>T$, this extends the
solution past its maximal lifespan, giving the desired contradiction. 

Next, using that $m(t) \to \infty$ as $t \uparrow T$ we prove that $T \geq 4/3$. Indeed, if $T<4/3$, the left inequality in \eqref{eq:reciprocal-initial} would give the uniform bound
\[
 m(t)\le\left(1-\frac34T\right)^{-1}, \quad \textup{for all } t < T,
\]
giving a contradiction. Hence, $4/3\le T\le2$.

Finally, integrating \eqref{eq:reciprocal} from $t$ to
$s\in(t,T)$, we get that
\[
 -\frac34(s-t)\le\frac1{m(s)}-\frac1{m(t)}
                  \le-\frac12(s-t).
\]
Thus, sending $s\uparrow T$ and using the fact that $1/m(s) \to 0$ as $s \uparrow T$, we infer that
\[
 \frac12(T-t)\le\frac1{m(t)}\le\frac34(T-t),
 \qquad
 \frac4{3(T-t)}\le m(t)\le\frac2{T-t}.
\]
Together with \eqref{eq:proof-mixture-derivatives}, this concludes the proof of Theorem \ref{thm:main} $(ii)$.
\end{proof}

We next deal with the proof of Theorem \ref{thm:main} $(iii)$. Let us start with a preliminary technical lemma.

\begin{lemma}\label{lem:deficit}
For $0 \leq t<T$ and $x>0$, put $P(t,x)=1-\theta(t,x)$.
Then
\begin{equation}\label{eq:deficitquotient}
 u(t,x)\le x^{1/4}P(t,x)^{3/8},\qquad
 P(t,x)\le(5/2)^{8/5}x^{2/5},\qquad
 u(t,x)\le(5/2)^{3/5}x^{2/5}.
\end{equation}
Moreover,
\begin{equation}\label{eq:derivative-deficit}
 0\le P_x(t,x)\le\frac{u(t,x)}x,\qquad
 |u_x(t,x)|\le\frac{u(t,x)}x.
\end{equation}
\end{lemma}

\begin{proof}
For $x>0$, probability normalization gives the formulas
\[
 P(t,x)=\int_{(0,\infty)}\frac{x^2}{a^2+x^2}\dd\mu_t(a),\qquad
 u(t,x)=\int_{(0,\infty)}\frac{ax}{a^2+x^2}\dd\mu_t(a).
\]
Their integrands show $0<P(t,x)<1$ and $u(t,x)>0$.
The quotient
$R(t,x):=u(t,x)x^{-1/4}P(t,x)^{-3/8}$ is therefore well defined and smooth
for $x > 0$. Moreover, observe that
\[
 DP=0,\qquad Dx=-u,\qquad Du=-\mathcal N(u).
\]
Consequently, by the lower bound in \eqref{eq:forcekernel},
\begin{equation}\label{eq:Ddeficit}
 DR(t,x)
 =\frac{-\mathcal N(u)(t,x)+u(t,x)^2/(4x)}{x^{1/4}P(t,x)^{3/8}}
 \le0,
\end{equation}

We justify integration along characteristics without approaching
the singular denominators at $x=0$. Fix $\tau<T$. On
$[0,\tau]$, $u$ is bounded by $1/2$ and is globally Lipschitz
in $x$ with constant at most $\max_{[0,\tau]}m$. Thus,
\[
 X_t(t,\xi)=-u(t,X(t,\xi)),\qquad X(0,\xi)=\xi,
\]
has a unique solution throughout $[0,\tau]$ for every $\xi$.
The bounded velocity prevents escape to infinity in finite time.
For $X>0$ the representation formula for $u$ implies $0<u(t,X)\le m(t)X$.
Hence, while the trajectory is positive, it follows that
\[
 -m(t)\le\frac{\partial_tX(t,\xi)}{X(t,\xi)}\le0.
\]
Integration yields
\begin{equation}\label{eq:deficit-characteristics}
 \xi\exp\left(-\int_0^t m(s)\dd s\right)
 \le X(t,\xi)\le\xi\qquad(\xi>0).
\end{equation}
The lower bound remains positive on $[0,\tau]$, so a first
arrival at zero is impossible. Thus every positive initial
point stays positive. Differentiating in $\xi$ also gives
\[
 \partial_\xi X(t,\xi)
 =\exp\left(-\int_0^t u_x(s,X(s,\xi))\dd s\right)>0.
\]
The bounds in \eqref{eq:deficit-characteristics} show that
$X(t,\xi)\to0$ as $\xi\downarrow0$ and
$X(t,\xi)\to\infty$ as $\xi\to\infty$.
For each fixed $t$, the map $\xi\mapsto X(t,\xi)$ is therefore
a bijection of $(0,\infty)$ onto itself.

Along such a trajectory, \eqref{eq:Ddeficit} gives
$\frac{\dd}{\dd t}R(t,X(t,\xi))\le0$. Initially,
\[
 P(0,\xi)=\frac{\xi^2}{1+\xi^2},\qquad
 u(0,\xi)=\frac{\xi}{1+\xi^2},
\]
so
\[
 R(0,\xi)
 =\frac{\xi/(1+\xi^2)}
        {\xi^{1/4}(\xi^2/(1+\xi^2))^{3/8}}
 =(1+\xi^2)^{-5/8}\le1.
\]
The bijection just proved gives $R(t,x)\le1$ at every $x>0$.
Since $\tau<T$ was arbitrary, this proves the first estimate
in \eqref{eq:deficitquotient} on the whole smooth lifespan.

We next verify the derivative bounds. Differentiating under the
integral sign is justified by the finite inverse first moment and the
kernel bounds in Lemma~\ref{lem:mixture}. Thus
\[
 P_x(t,x)=\int_{(0,\infty)}\frac{2a^2x}{(a^2+x^2)^2}\dd\mu_t(a),\qquad
 u_x(t,x)=\int_{(0,\infty)}\frac{a(a^2-x^2)}{(a^2+x^2)^2}\dd\mu_t(a).
\]
Having this representation at hand, \eqref{eq:derivative-deficit} immediately follows from the following bounds
\[
 \frac{2a^2x}{(a^2+x^2)^2}\le\frac{a}{a^2+x^2} \quad \textup{and} \quad
 \left|\frac{a(a^2-x^2)}{(a^2+x^2)^2}\right|
 \le\frac{a}{a^2+x^2}, \quad \textup{for } a,x > 0.
\]
 
We are just missing to show the second and third bounds in \eqref{eq:deficitquotient}. Combining the first inequalities in \eqref{eq:deficitquotient} and \eqref{eq:derivative-deficit} gives
\[
 \partial_x(P(t,x)^{5/8})
 =\frac58P(t,x)^{-3/8}P_x(t,x)\le\frac58x^{-3/4}, \quad \textup{for all } x > 0. 
\]
Thus, for $0<\varepsilon<x$, integration gives
\[
 P(t,x)^{5/8}-P(t,\varepsilon)^{5/8}
 \le\frac58\int_\varepsilon^x y^{-3/4}\dd y
 =\frac52(x^{1/4}-\varepsilon^{1/4}).
\]
Since $\theta(t,0)=1$ and $\theta$ is continuous,
$P(t,\varepsilon)\to0$ as $\varepsilon\downarrow0$.
Thus, the second inequality in \eqref{eq:deficitquotient} holds. By substituting into the first, we then obtain the remaining one and conclude the proof.  
\end{proof}

\begin{proof}[\textbf{Proof of Theorem \ref{thm:main} -- \ref{itemthree}}]
By \eqref{eq:deficitquotient}, \eqref{eq:derivative-deficit}, and parity, it follows that
\[
\max\{ |\theta_x(t,x)|,\ |u_x(t,x)| \}
 \le (5/2)^{3/5}|x|^{-3/5}, \quad \textup{for all } x \neq 0.
\]
For every interval $I$ of length $h>0$, the even function
$|x|^{-3/5}$ has its largest integral on $[-h/2,h/2]$. Consequently, for every interval $I$ of length $h > 0$, 
\[
 \int_I |x|^{-3/5}\,\dd x
 \le\frac{2^{3/5}}{2/5}h^{2/5}.
\]
Thus, for any $x < y$ and either $f = \te(t,\cdot)$ or $f = u(t,\cdot)$, it follows that
$$
|f(y)-f(x)| \leq \Big(\frac{5}{2}\Big)^{\frac35}  \int_{x}^y |z|^{-3/5} \dd z \leq  \frac{5^{8/5}}{2} |x-y|^{\frac25}.
$$
This concludes the proof. 
\end{proof}

The slope $m(t)=u_x(t,0)=H\theta_x(t,0)$ is a nonlocal quantity. Its growth does not directly give a lower bound for
$\norm{\theta_x(t,\cdot)}_{L^\infty}$. We obtain such a bound by combining
the inverse-moment identity
\begin{equation} \label{eq:inversemoment}
m(t) =\frac{2}{\pi}	\int_0^{\infty} \frac{P(t,x)}{x^2}\, \dd x,
\end{equation}
which immediately follows from the definitions of $P$ and $m$, with a suitable control of $\te_{xx}(t,0)$.

\begin{lemma}\label{lem:curvature}
For every $0 \leq t<T$, define
$
K(t):=-\theta_{xx}(t,0).
$
Then, $K(t)$ is finite and
\begin{equation*} 
\frac{K(t)}{2}=\int_{(0,\infty)} a^{-2}\dd\mu_t(a),\quad
2 m(t)^2\le K(t),\quad
 0\le 2 P(t,x)\le K(t)x^2, \quad \textup{for all } (t,x) \in [0,T) \times \R.
\end{equation*}
Moreover,
\begin{equation}\label{eq:Knorm}
 \norm{\theta_{xx}(t,\cdot)}_{L^\infty}=K(t),
\end{equation}
and
\begin{equation}\label{eq:KODE}
 K'(t)=2m(t)K(t),\quad K(0)=2.
\end{equation}
In particular,
\begin{equation}\label{eq:Krate}
 K(t)=2\exp\left(2\int_0^t m(s)\dd s\right)
 \le2\left(\frac{T}{T-t}\right)^4
 \le\frac{32}{(T-t)^{4}}.
\end{equation}
\end{lemma}

\begin{proof}
Let $t \in [0,T)$ be fixed but arbitrary. The probability measure representation of $\te$ gives, for $x > 0$,
\begin{equation} \label{eq:Px2}
 \frac{P(t,x)}{x^2}
 =\int_{(0,\infty)}\frac{\dd\mu_t(a)}{a^2+x^2}.
\end{equation}
As $x\downarrow0$, the integrand increases to $a^{-2}$.
Monotone convergence identifies the limit of the right hand side
with $\int_{(0,\infty)} a^{-2}\dd\mu_t$, allowing initially the value
$+\infty$. On the other hand, $\theta(t,0)=1$,
$\theta_x(t,0)=0$, and Taylor expansion at zero give
\[
 P(t,x)=-\frac12\theta_{xx}(t,0)x^2+o(x^2).
\]
The limit of the left hand side in \eqref{eq:Px2} is therefore the finite number $-\theta_{xx}(t,0)/2$. This both proves finiteness of the
inverse second moment and identifies it with $K(t)/2$.
Since $\mu_t$ has mass one, Cauchy--Schwarz gives
$2 m(t)^2 \leq K(t)$. Also, using that $(a^2+x^2)^{-1}\le a^{-2}$ in \eqref{eq:Px2}, gives $2 P(t,x)\le K(t)x^2$ for all $x \in \R$.

Direct differentiation of the kernel $p_a(x)$ gives
\[
 p_a''(x)=\frac{2a^2(3x^2-a^2)}{(a^2+x^2)^3}.
\]
Thus, it follows that $|p_a''(x)|\le2a^{-2}$. This bound is integrable against
$\mu_t$. Hence, differentiating under the integral sign  gives
$$
\theta_{xx}(t,x)=\int_{(0,\infty)} p_a''(x)\dd\mu_t(a) \quad \textup{and} \quad |\theta_{xx}(t,x)|\le K(t).$$ 
Since $p_a''(0)=-2a^{-2}$, equality holds in absolute value
there and \eqref{eq:Knorm} follows.

Now, differentiating twice in $x$ the transport equation in \eqref{eq:CCF}, we get that
\[
 \theta_{xxt}
 =u_{xx}\theta_x+2u_x\theta_{xx}+u\theta_{xxx}. 
\]
At $x = 0$, the first and third terms on the right hand side vanish because
$\theta_x(t,0)=u(t,0)=0$. Thus, it follows that
$\theta_{xxt}(t,0)=2m(t)\theta_{xx}(t,0)$. Likewise, the initial datum $\theta_0$ satisfies $\theta_0''(0) = -2$. Hence, \eqref{eq:KODE} follows. Moreover, solving this ODE gives the exponential formula in \eqref{eq:Krate}.

Finally, using the pointwise upper bound for $m(t)$ already proven in Theorem \ref{thm:main} $(ii)$, we get that 
\[
 2\int_0^t m(s)\dd s
 \le4\int_0^t\frac{\dd s}{T-s}
 =4\log\frac{T}{T-t}.
\]
Exponentiating and using $T\le2$ gives the two claimed
upper bounds for $K(t)$. This concludes the proof.
\end{proof}

\begin{proof}[\textbf{Proof of Theorem \ref{thm:main} -- \ref{itemfour}}]
We set $G(t) := \norm{\theta_x(t,\cdot)}_{L^\infty}$, and prove that
\begin{equation}\label{eq:logcomparison}
 m(t)\le\frac3\pi+\frac2\pi G(t)\log(1+K(t)).
\end{equation}
Once we have this inequality, the proof is essentially straightforward. We provide the details for completeness.  

We fix an arbitrary $t \in [0,T)$ and suppress time in $m,G,K$ and $P$.
By Lemma \ref{lem:curvature} we know that  $P(x)\le Kx^2/2$.
The fundamental theorem of calculus and $\theta(0)=1$ give
$P(x)=|\theta(0)-\theta(x)|\le Gx$ for $x>0$, while
Theorem \ref{thm:main} $(i)$ gives $P(x)\le1$.
Thus, we have that
\[
 0\le P(x)\le\min\{Kx^2/2,Gx,1\}.
\]

We now set $r_0:=(1+K)^{-1}$. Since $K>0$,
$0<r_0<1$, so the identity
\eqref{eq:inversemoment} can be split as
\begin{align*}
 m
 &=\frac2\pi\left(\int_0^{r_0}\frac{P(x)}{x^2}\dd x
                +\int_{r_0}^1\frac{P(x)}{x^2}\dd x
                +\int_1^\infty\frac{P(x)}{x^2}\dd x\right)\\
 &\le\frac2\pi\left(\int_0^{r_0}\frac K2\dd x
                +\int_{r_0}^1\frac Gx\dd x
                +\int_1^\infty\frac{\dd x}{x^2}\right) =\frac2\pi\left(\frac{K}{2(1+K)}
                   +G\log(1+K)+1\right) \le\frac3\pi+\frac2\pi G\log(1+K).
\end{align*}
This proves \eqref{eq:logcomparison}. 

Now, let $\delta=T-t$ and assume $0<\delta\le1/2$.
The lower bound for $m(t)$ in Theorem \ref{thm:main} $(ii)$ implies that
\[
 G(t)\log(1+K(t))
 \ge\frac\pi2m(t)-\frac32
 \ge\frac{2\pi}{3\delta}-\frac32
 \ge\frac{\pi}{3\delta}.
\]
Also, by \eqref{eq:Krate},
\[
 \log(1+K(t))\le\log(1+32\delta^{-4})
 \le\log(81\delta^{-4})=4\log(3/\delta).
\]
Hence, it follows that
\[
 G(t)\ge\frac{\pi}{12\delta\log(3/\delta)}.
\]
The upper bound we have is
$G(t)\le m(t)\le2/\delta$, which follows from
\eqref{eq:proof-mixture-derivatives} and the upper bound for $m(t)$ given in Theorem \ref{thm:main} $(ii)$. Finally,
$\delta\log(3/\delta)\to0$ as $\delta\downarrow0$,
so the lower bound proves the full limit
$G(t)\to\infty$, not only unboundedness along a sequence.
\end{proof}

In order to conclude the proof of Theorem \ref{thm:main}, we are just missing to show the H\"older seminorm blow up claimed in $(v)$. We divide the proof into several lemmas. However, before doing so, let us explicitly point out that, if one does not want to include the endpoint $\beta = 2/3$ and just look for unboundedness whenever $\be \in (2/3,1)$, the proof can be slightly simplified. 

First, we quantify how the positive term $I$ in the equation for $m$ improves the relation between $m$ and $K$. More precisely, we establish a lower bound for $I$ depending only on $m$ and $K/2$.  

\begin{lemma} \label{lem:log-remainder}
Let $\mu$ be a Borel probability measure on $(0,\infty)$ with
finite moments $m=\int_{(0,\infty)} a^{-1}\dd\mu(a)$ and $J=\int_{(0,\infty)} a^{-2}\dd\mu(a)$.
Then $m>0$, $J\ge m^2$, and
\begin{equation}\label{eq:log-remainder}
 I:=\iint_{(0,\infty)^2}\frac{\dd\mu(a)\dd\mu(b)}{(a+b)^2}
 \ge\frac{m^2}{8\log(16J/m^2)}.
\end{equation}
\end{lemma}

\begin{proof}
For all $a,s > 0$, set $\cX(a)=a^{-1}$ and $\cF(s):=\mu\{a\in(0,\infty):\cX(a)>s\}.$
An elementary layer-cake argument gives
\[
 m=\int_0^\infty \cF(s)\,\dd s,
 \qquad
 J=2\int_0^\infty s\cF(s)\,\dd s.
\]
In particular, $m>0$ and $m^2\le J$ by Cauchy--Schwarz. Moreover, using that
\[
 \frac1{a+b}
 =\frac{\cX(a)\cX(b)}{\cX(a)+\cX(b)}
 \ge \frac12\min\{\cX(a),\cX(b)\},
\]
another layer-cake argument yields
\[
 I
 \ge \frac14
 \iint_{(0,\infty)^2} \min\{\cX(a),\cX(b)\}^2\,\dd\mu(a)\dd\mu(b)
 =\frac12\int_0^\infty s\cF(s)^2\,\dd s.
\]

Now, set $\ell := m/4$ and $L := 4J/m$. Since $\cF\le1$, it follows that
\[
 \int_0^\ell \cF(s)\,\dd s\le\frac m4,
\]
while
\[
 \int_L^\infty \cF(s)\,\dd s
 =
 \int_{(0,\infty)}(\cX(a)-L)_+\,\dd\mu(a)
 \le
 \frac1L\int_{(0,\infty)}\cX(a)^2\,\dd\mu(a)
 =\frac m4.
\]
Hence, it follows that
\[
 \int_\ell^L \cF(s)\,\dd s\ge\frac m2.
\]
By Cauchy--Schwarz, 
\[
 \frac{m^2}{4}
 \le
 \left(\int_\ell^L s\cF(s)^2\,\dd s\right)
 \left(\int_\ell^L\frac{\dd s}{s}\right)
 \le
 2I\log\frac{L}{\ell}
 =
 2I\log\frac{16J}{m^2}.
\]
This proves \eqref{eq:log-remainder}.
\end{proof}
 
\begin{lemma}\label{lem:holder-lower}
For each $0<\beta<1$, there exists $c_\be > 0$ such that
\begin{equation}\label{eq:holder-lower}
 \min\{[\theta(t,\cdot)]_{C^\beta},[u(t,\cdot)]_{C^\beta}\}
 \ge c_\beta\frac{m(t)^{2-\beta}}{K(t)^{1-\beta}}, \quad \textup{for all } t \in [0,T).
\end{equation}
\end{lemma}

\begin{proof}
We fix an arbitrary $0 \leq t<T$ and suppress it in the notation. For the active scalar $\theta$, we set $A_\beta:=[\theta]_{C^\beta}$ and stress that
it is positive because $\theta(0)=1$ while
$\theta(x)\le(1+x^2)^{-1}<1$ for $x\ne0$.
The second-moment estimate in Lemma \ref{lem:curvature} and the definition of the H\"older seminorm give
\[
 0\le 2P(x)\le K x^2,\qquad P(x)\le A_\beta x^\beta, \quad \textup{for } x > 0. 
\]
The two upper bounds coincide at the positive radius
$r=(2A_\beta/K)^{1/(2-\beta)}$. Splitting then \eqref{eq:inversemoment} there, gives
\begin{align*}
 m =\frac2\pi\int_0^\infty\frac{P(x)}{x^2}\dd x \le\frac2\pi\left(\int_0^r \frac{K}{2} \dd x+\int_r^\infty A_\beta x^{\beta-2}\dd x\right)=\frac2\pi\left(\frac{K}{2}r+\frac{A_\beta r^{\beta-1}}{1-\beta}\right).
\end{align*}
Also, by the choice of $r$,
$A_\beta r^{\beta-1}=Kr/2 =A_\beta^{1/(2-\beta)}(K/2)^{(1-\beta)/(2-\beta)}$.
Therefore, it follows that
\[
 m\le\frac{2(2-\beta)}{\pi(1-\beta)}
       A_\beta^{1/(2-\beta)}\Big(\frac{K}{2} \Big)^{(1-\beta)/(2-\beta)}.
\]
The lower bound for the active scalar in \eqref{eq:holder-lower} immediately follows by rearranging this inequality.

For the field $u$, the probability measure representation and the finiteness of
$m$ give, at every $x>0$,
\begin{align*}
 m-\frac{u(x)}x =\int_{(0,\infty)}\left(\frac1a-\frac{a}{a^2+x^2}\right)\dd\mu_t(a)=\int_{(0,\infty)}\frac{x^2}{a(a^2+x^2)}\dd\mu_t(a)
 \le\frac x2\int_{(0,\infty)} a^{-2}\dd\mu_t(a)=\frac{Kx}{4}.
\end{align*}
Setting then $x_*:=2m/K>0$, we get that
$$
\frac{u(x_*)}{x_*}\ge m-\frac{Kx_*}{4}=\frac{m}{2}.
$$
Since $u(0)=0$ and $u(x_*)>0$, we conclude that
\[
 [u]_{C^\beta}
 \ge\frac{|u(x_*)-u(0)|}{x_*^\beta}
 =\frac{u(x_*)}{x_*^\beta}
 \ge\frac m2x_*^{1-\beta}
 =\frac{m^{2-\beta}}{2^{\beta} K^{1-\beta}}.
\]
\end{proof}

\begin{lemma}\label{lem:log-curvature}
Let $\cL(z):= \log(4z)/\log(4).$ For every $0 \leq t < T$, it follows that
\begin{equation}\label{eq:log-curvature}
 K(t)\le 2 m(t)^4 \cL(m(t))^{-1/3}.
\end{equation}
\end{lemma}

\begin{proof}
For simplicity, we set $J(t):= K(t)/2$. The function $J$ and $m$ are positive and continuously
differentiable, with $J(0)=m(0)=1$. Combining
\eqref{eq:mODE} and \eqref{eq:KODE} gives
\begin{equation}\label{eq:curvature-quotient}
 \frac{\dd}{\dd t}\log\frac{J(t)}{m(t)^4}
 =\frac{J'(t)}{J(t)}-4\frac{m'(t)}{m(t)}
 =2m(t)-4\left(\frac{m(t)}{2}+\frac{I(t)}{m(t)} \right)
 =-\frac{4I(t)}{m(t)}\le0.
\end{equation}
Since $J(0)/m(0)^4=1$, integration first gives
$J(t)\le m(t)^4$. This preliminary bound is used to
improve the right side of the same identity. 

Lemma~\ref{lem:log-remainder} applies to $\mu_t$, because
Lemma~\ref{lem:curvature} proved its finite inverse second
moment. Since $m(t)\ge1$ and $m(t)^2\le J(t)\le m(t)^4$, all the
logarithms below are positive and
\[
 \log(16J(t)/m(t)^2)\le\log(16m(t)^2)=2\log(4m(t)).
\]
Taking reciprocals therefore gives
\[
 I(t)\ge\frac{m(t)^2}{8\log(16J(t)/m(t)^2)}
   \ge\frac{m(t)^2}{16\log(4m(t))},
\]
and substituting in \eqref{eq:curvature-quotient} yields
\[
 \frac{\dd}{\dd t}\log\frac{J(t)}{m(t)^4}
 \le-\frac{m(t)}{4\log(4m(t))}.
\]
To integrate in the variable $m$, note that
$m'(t)\ge m(t)^2/2>0$, $m(0)=1$, and $m(t)\to\infty$.
Hence $t\mapsto m(t)$ is a continuously differentiable
bijection from $[0,T)$ onto $[1,\infty)$, with a
continuously differentiable inverse on $(1,\infty)$.
Write this inverse as $t=t(z)$ and put
$\cJ(z):=J(t(z))/z^4$. The chain rule and
$m'(t)\le3m(t)^2/4$ give
\[
 \frac{\dd}{\dd z}\log \cJ(z)
 \le-\frac{z}{4\log(4z)}\frac1{m'(t(z))}
 \le-\frac1{3z\log(4z)}.
\]
Integrating from $1$ to $m(t)$, and using that $\cJ(1)=1$, gives
\begin{align*}
 \log\frac{J(t)}{m(t)^4} \le-\frac13\int_1^{m(t)}\frac{\dd z}{z\log(4z)}=-\frac13\bigl(\log(\log(4m(t)))-\log(\log4)\bigr)=-\frac13\log \cL(m(t)).
\end{align*}
Exponentiation proves \eqref{eq:log-curvature}.
\end{proof}

\begin{proof}[\textbf{Proof of Theorem~\ref{thm:main} -- \ref{itemfive}}] 
Let $\cL$ be as in Lemma \ref{lem:log-curvature}. Combining Lemma \ref{lem:holder-lower} applied with $\be = 2/3$ and Lemma \ref{lem:log-curvature}, we get the existence of $c > 0$ such that
\begin{equation}\label{eq:holder-divergence-rate}
 \min\{[\theta(t, \cdot)]_{C^{2/3}},[u(t,\cdot)]_{C^{2/3}}\}
 \ge c \cL(m(t))^{1/9}, \quad \textup{for all } t \in [0,T). 
\end{equation}
Since $m(t) \to \infty$ as $t \uparrow T$, and $\cL(z) \to +\infty$ as $z \to \infty$, we have that $\cL(m(t))^{1/9} \to \infty$ as $t \uparrow T$, and the conclusion follows. 
\end{proof}

\section*{Acknowledgments}
A.C. and A.E. acknowledge financial support from the Severo Ochoa Programme for Centers of Excellence (Grants \nolinkurl{CEX2019-000904-S} and \nolinkurl{CEX-2023-001347-S}), funded by 
\nolinkurl{MCIN/AEI/10.13039/501100011033}. A.C. was also supported by the grants \nolinkurl{PID2020-114703GB-I00}, \nolinkurl{PID2024-158418NB-I00}, and \nolinkurl{RED2024-153842-T}, all funded by \nolinkurl{MICIU/AEI/10.13039/501100011033}. This work has received funding from the European Research Council (ERC) under the European Union's Horizon 2020 research and innovation programme through the grant agreement \nolinkurl{862342} (A.E.). A.E. is also partially supported by the grant \nolinkurl{PID2025-168410NB-I00} of the Spanish Science Agency. A.J.F. is partially supported by the grants \nolinkurl{PID2023-149451NA-I00} of \nolinkurl{MCIN/AEI/10.13039/501100011033/FEDER, UE;}  and \nolinkurl{RYC2024-049142-I}, \nolinkurl{MICINN} (Spain).

\appendix
\section{The tame transport estimate}\label{app:energy}

\begin{lemma}\label{lem:energy}
Let $f,w$ be smooth and decaying, with
$f_t=w f_x$, and let $k\ge2$ be an integer. Then
\begin{equation}\label{eq:tame}
 \frac{\dd}{\dd t}\norm{f(t,\cdot)}_{H^k}
 \le C_k\left(
   \norm{w_x(t,\cdot)}_{L^\infty}\norm{f(t,\cdot)}_{H^k}
   +\norm{w(t,\cdot)}_{H^k}\norm{f_x(t,\cdot)}_{L^\infty}
 \right).
\end{equation}
\end{lemma}
\begin{proof}
The estimate follows by integration by parts and a standard
integer-order commutator estimate. See for instance
\cite[Lemma 3.4]{MajdaBertozzi}.
\end{proof}

\end{document}